\documentclass[letterpaper, 11pt,  reqno]{amsart}

\usepackage{amsmath,amssymb,amscd,amsthm,amsxtra, esint}
\usepackage{skak}
\usepackage{transparent}

\usepackage{enumerate}

\allowdisplaybreaks[2]

\usepackage{color}

\definecolor{gr}{rgb}   {0.,   0.69,   0.23 }
\definecolor{bl}{rgb}   {0.,   0.5,   1. }
\definecolor{mg}{rgb}   {0.85,  0.,    0.85}
\definecolor{yl}{rgb}   {0.8,  0.7,   0.}
\definecolor{or}{rgb}  {0.7,0.2,0.2}

\newtheorem{theorem}{Theorem} [section]

\newtheorem{lemma}[theorem]{Lemma}
\newtheorem{proposition}[theorem]{Proposition}
\newtheorem{remark}[theorem]{Remark}

\newtheorem*{acknowledgment}{Acknowledgments}

\DeclareMathOperator*{\intt}{\int}

\newcommand{\noi}{\noindent}
\newcommand{\Z}{\mathbb{Z}}
\newcommand{\R}{\mathbb{R}}
\newcommand{\C}{\mathbb{C}}
\newcommand{\T}{\mathbb{T}}

\let\Re=\undefined\DeclareMathOperator*{\Re}{Re}

\let\P= \undefined
\newcommand{\P}{\mathbf{P}}

\newcommand{\FL}{\mathcal{F}L}

\newcommand{\N}{\mathcal{N}}

\def\norm#1{\|#1\|}

\newcommand{\al}{\alpha}

\newcommand{\dl}{\delta}

\newcommand{\eps}{\varepsilon}

\newcommand{\g}{\gamma}
\newcommand{\G}{\Gamma}
\newcommand{\ld}{\lambda}

\newcommand{\s}{\sigma}

\newcommand{\ft}{\widehat}

\newcommand{\cj}{\overline}
\newcommand{\dx}{\partial_x}
\newcommand{\dt}{\partial_t}

\newcommand{\ta}{\theta}

\renewcommand{\l}{\ell}

\renewcommand{\O}{\Omega}

\newcommand{\les}{\lesssim}
\newcommand{\ges}{\gtrsim}

\newcommand{\jb}[1]
{\langle #1 \rangle}

\newcommand{\ind}{\mathbf 1}

\numberwithin{equation}{section}
\numberwithin{theorem}{section}

\begin{document}
\baselineskip = 12.7pt

\title[Transport of Gaussian measures under 1-d fractional NLS]
{On the transport of Gaussian measures under the one-dimensional
fractional nonlinear Schr\"odinger equations}

\author[Justin Forlano and William J. Trenberth]
{Justin Forlano and William J. Trenberth}

\address{
 Justin Forlano, Maxwell Institute for Mathematical Sciences\\
 Department of Mathematics\\
 Heriot-Watt University\\
 Edinburgh\\ 
 EH14 4AS\\
  United Kingdom}

\email{j.forlano@hw.ac.uk}

\address{William J. Trenberth,
Maxwell Institute for Mathematical Sciences \\
School of Mathematics \\
University of Edinburgh \\
Edinburgh \\
UK \\
EH9 3FD}

\email{W.J.Trenberth@sms.ed.ac.uk} % CHECK EMAIL TO SUE


\begin{abstract}
Under certain regularity conditions, we establish quasi-invariance of Gaussian measures on periodic functions under the flow of cubic fractional nonlinear Schr\"{o}dinger equations on the one-dimensional torus.

\end{abstract}

%\date{\today}

\vspace*{-5mm}

\maketitle
\tableofcontents

\vspace*{-6mm}

\section{Introduction}

\subsection{Cubic fractional nonlinear Schr\"odinger equations}
We consider the cubic fractional nonlinear Schr\"{o}dinger equation (FNLS) 
on the one-dimensional torus $\T=\R/(2\pi\Z)$:
\begin{equation}\label{FNLS}
\begin{cases}
i\partial_t u+(-\dx^2)^\alpha u=\pm |u|^2u,\\
u|_{t=0}=u_0,
\end{cases}
\end{equation}
where $u:\R \times \T \longmapsto \mathbb{C}$ is the unknown function. For $\al>0$, let $(-\dx^{2})^{\al}$ be the Fourier multiplier operator defined by $ ((-\dx^{2})^{\al}f )\,\,\widehat{}\,\,(n):=|n|^{2\al}\ft f(n)$, $n\in \Z$, where $\ft f$ denotes the Fourier transform of $f$.
We say \eqref{FNLS} is defocusing when the sign on the nonlinearity is positive and focusing when the sign on the nonlinearity is negative. 

The equation FNLS~\eqref{FNLS} arises in various physical settings.
When $\al=1$, the equation~\eqref{FNLS} is the cubic nonlinear Schr\"{o}dinger equation (NLS) which appears as a model in the study of nonlinear optics, fluids and plasma physics; see \cite{Sulem} for a general survey. For $\al=2$, \eqref{FNLS} corresponds to the cubic fourth order NLS (4NLS) and has applications to the study of solitons in magnetic materials~\cite{4NLS1, 4NLS2}. FNLS with non-local dispersion $\tfrac{1}{2}< \al<1$ arises in continuum limits of long-range lattice interactions~\cite{FRAC1}. When $\al \leq \tfrac{1}{2}$, the equation FNLS~\eqref{FNLS} is no longer dispersive. 
However, the cubic nonlinear half-wave equation, corresponding to $\al=\tfrac 12$, 
emerges in the study of wave turbulence~\cite{FRAC3, FRAC2}, gravitational collapse~\cite{FRAC4, FRAC5} and has been well-studied analytically~\cite{GG,KLP,GLPR}. We also refer to \cite{Thirouin} for a study of \eqref{FNLS} when $\frac{1}{3}<\al <\frac{1}{2}$.
In the following, we focus on studying FNLS~\eqref{FNLS} for $\al>\tfrac 12$ where dispersion is present.

 The well-posedness theory of \eqref{FNLS} in the $L^{2}$-based Sobolev spaces $H^{\s}(\T)$ crucially depends upon the strength of the dispersion; namely, if $\al\geq 1$ or $\tfrac{1}{2}<\al<1$.
In \cite{Bou}, Bourgain proved local well-posedness of NLS in $L^{2}(\T)$, which immediately extends to global well-posedness in $L^{2}(\T)$ as a consequence of mass conservation:  
\begin{equation}
M(u)(t)=\int_{\T}|u(t,x)|^2\,dx=M(u)(0) \quad \text{for all} \,\, t\in \R. \label{mass}
\end{equation}
A persistence-of-regularity argument then implies global well-posedness of NLS in $H^{\s}(\T)$ for any $\s\geq 0$. This result is sharp in the sense that NLS is ill-posed if $\s<0$. More precisely, the solution map\footnote{For clarity of presentation, we neglect to explicitly show the dependence of the solution map $\Phi$ on $\al$. Unless otherwise stated, the precise value of $\al$ will be clear from the context.} $\Phi: u_0\in H^{\s}(\T) \mapsto u\in C([-T,T];H^{\s}(\T))$, if it even exists in view of the non-existence of solutions in~\cite{GO}, is discontinuous~\cite{Molinet} (see also \cite{BGT, OW1, ChofPoc,Oh, Kishimoto}). In \cite[Appendix A]{oh2016quasi}, the 4NLS was shown to be globally well-posed in $H^{\s}(\T)$ for any $\s\geq 0$.
 In Appendix~\ref{app:gwp} of this paper, we extend this global well-posedness result to FNLS~\eqref{FNLS} for any $\al>1$. 
 
The well-posedness situation for FNLS~\eqref{FNLS} is somewhat less complete in the setting $\tfrac{1}{2}<\al<1$. In~\cite{Cho}, Cho, Hwang, Kwon and Lee proved local well-posedness of FNLS~\eqref{FNLS} in $H^{\s}(\T)$ for $\s\geq \tfrac{1-\al}{2}$ by a contraction mapping argument; see also \cite{demirbas2013existence}. Following the argument in~\cite[Appendix A]{oh2016quasi}, we can show the solution map for FNLS~\eqref{FNLS} fails to be locally uniformly continuous in $H^{\s}(\T)$ for any $\s<0$. Thus we can not construct solutions below $L^{2}(\T)$ using a contraction mapping argument. See also \cite{ChofPoc}.

As for global well-posedness, the flow of FNLS~\eqref{FNLS} conserves the energy $H(u)$; that is,
\begin{equation} 
H(u)(t)=\frac{1}{2}\int_{\T} \left | (-\dx^2)^{\frac{\al}{2}} u(t,x) \right|^2\,dx \pm \frac{1}{4}\int_{\T}|u(t,x)|^4\,dx=H(u)(0). \label{energy}
\end{equation}
In the defocusing case, the energy
controls the $H^{\al}$-norm and hence energy conservation can be used to globalise (in time) all solutions with regularity at or above the energy space, that is, for when $\s\geq \al$. This result also holds in the focusing case, as the $H^{\al}$-norm can be controlled in terms of both the energy and the mass by using the Gagliardo-Nirenberg inequality 
\begin{align}
\|u\|_{L^4}^{4} \les \| (-\dx^2)^{\tfrac{\al}{2}}u\|_{L^{2}}^{\frac{1}{\al}}\|u\|_{L^2}^{4-\frac{1}{\al}} \label{GNineq}
\end{align}    
for $\al\geq \tfrac 14$.
Below the energy space $H^{\al}(\T)$, global well-posedness of FNLS \eqref{FNLS} (for both defocusing and focusing nonlinearities) in $H^{\s}(\T)$ was obtained for $\s>\tfrac{10\al+1}{12}$~\cite{demirbas2013existence} by using the high-low frequency decomposition of Bourgain~\cite{Bo98}.

In summary, the flow of FNLS~\eqref{FNLS} is well-defined under the following conditions:

\begin{proposition}[Well-posedness of the cubic FNLS~\eqref{FNLS} in $H^{\s}(\T)$~\cite{Bou, Cho, demirbas2013existence, oh2016quasi}] \label{prop: gwp} \
\begin{enumerate}[\normalfont (i)]
\setlength\itemsep{0.3em}
\item Let $\al\geq 1$. Then, the cubic FNLS~\eqref{FNLS} is globally well-posed in $H^{\s}(\T)$ for $\s \geq 0$. 
\item  Let $\tfrac 12<\al <1$.  Then, the cubic FNLS~\eqref{FNLS} is locally well-posed in $H^{\s}(\T)$ for $\s\geq \frac{1-\al}{2}$. Moreover, the cubic FNLS~\eqref{FNLS} is
 globally well-posed for $\s>\frac{10\al+1}{12}$.
\end{enumerate}
\end{proposition}

In the following, we make no distinction between the defocusing or focusing nature of \eqref{FNLS} and henceforth we assume that \eqref{FNLS} is defocusing. 
For future use, we define $\Phi(t): u_0\in H^{\s}(\T) \mapsto u(t)\in H^{\s}(\T)$ to be the solution map of FNLS~\eqref{FNLS} (when it exists) at time $t$.

\subsection{Main result }
Our goal in this paper is to study the transport property of Gaussian measures on periodic functions under the flow of FNLS~\eqref{FNLS}.
Given $s\in \R$, we define the Gaussian measure $\mu_s$ to be the induced probability measure under the map\footnote{From now on, we drop the factor $2\pi$ as it plays no role in our analysis.}:
\begin{equation}
\omega\in\Omega\longmapsto u^\omega(x)=\sum\limits_{n\in\Z}\frac{g_n(\omega)}{\langle n\rangle^s}e^{inx},
\label{gaussiandata}
\end{equation}
where $\jb{\, \cdot \, }:=(1+|\cdot |^2)^{\frac{1}{2}}$ and $\{g_n\}_{n\in\Z}$ is a sequence of independent standard complex-valued Gaussian random variables, i.e $\textnormal{Var}(g_n)=2$, on a probability space $(\O,\mathcal{F},\mathbb{P})$. Formally, $\mu_s$ has density
\begin{equation*}
d\mu_s=Z^{-1}_se^{-\frac{1}{2}\norm{u}^2_{H^s}}du=Z^{-1}_s\prod\limits_{n\in\Z}e^{-\frac{1}{2}\langle n\rangle^{2s}|\widehat{u}_n|^2}d\widehat{u}_n.
\end{equation*}
A computation shows that the random distribution \eqref{gaussiandata} belongs to $H^{s-\frac{1}{2}-\eps}(\T)$ and not to $H^{s-\frac 12}(\T)$ almost surely. It follows that the Gaussian measure $\mu_s$ is supported on $H^{s-\frac{1}{2}-\eps}(\T)\setminus H^{s-\frac 12}(\T)$ for any $\eps>0$.
 Thus, in order to discuss the transport property of these measures under the flow of \eqref{FNLS}, Proposition~\ref{prop: gwp} restricts us to the range:
\begin{align}
s>\max\bigg( \frac{1}{2}, 1-\frac{\al}{2} \bigg),  \label{soptimal}
\end{align}
which ensures there exists well-defined dynamics within the support of $\mu_s$.
We now state the definition of quasi-invariant measures: given a measure space $(X,\mu)$, we say that $\mu$ is quasi-invariant under a transformation $T:X\rightarrow X$ if the push-forward measure $T_{*}\mu=\mu\circ T^{-1}$ and $\mu$ are mutually absolutely continuous with respect to each other.

The quasi-invariance of Gaussian measures supported on periodic functions under the flow of 4NLS was recently studied by Oh and Tzvetkov~\cite{oh2016quasi} and Oh, Sosoe and Tzvetkov~\cite{OhTzvetSos}. See also Remark~\ref{Remark: Bou, Zhis quasi result} and the recent work~\cite{OhTsutTzvet} on Schr\"odinger-type equations.
Our main goal is to extend these quasi-invariance results to more general values of dispersion $\al$.
Thus, in this direction we establish the following:
\begin{theorem}\label{Main result}
Let $s\in\R$ and $\al >\tfrac{1}{2}$ be such that
\begin{enumerate}[\normalfont (i)]
\setlength\itemsep{0.3em}
\item $s>\max\big( \frac{2}{3},\frac{11}{6}-\al \big)$   if $\al\geq 1$, or    \label{resultals}
\item $s>\frac{10\alpha+7}{12}$ if  $\tfrac{1}{2}<\al<1$. \label{resultals1}
\end{enumerate}
Then, the Gaussian measure $\mu_s$ is quasi-invariant under the flow of the cubic FNLS \eqref{FNLS}. More precisely, given any measurable set $A\subset H^{s-\frac 12-\eps}(\T)$ satisfying   $\mu_{s}(A)=0$, we have $\mu_{s}(\Phi(-t)(A))=0$ for every $t\in \R$.
\end{theorem}

Our proof of quasi-invariance of $\mu_s$ in the case $\tfrac 12<\al<1$ also holds for any 
\begin{align}
\min\bigg( 1, \frac{11}{6}-\al \bigg) <s\leq \frac{10\al+7}{12}. \label{locals}
\end{align}
The restriction in Theorem~\ref{Main result}~\eqref{resultals1} is due to a lack of globally well-defined dynamics within the support of $\mu_s$ for $s$ satisfying \eqref{locals} (see Proposition~\ref{prop: gwp}~(ii)). However, our arguments in this paper allow us to recover the following \emph{local-in-time quasi-invariance} result: 
 
\begin{theorem}[Local-in-time quasi-invariance] \label{thm:localqi}
Let $\frac 12 < \al <1$ and $s$ satisfy \eqref{locals}. Then, for every $R>0$, there exists $T>0$ such that
for every measurable
\begin{align*}
A \subset \{ u\in H^{s-\frac 12 -\eps}(\T)\, : \, \|u\|_{H^{s-\frac 12 -\eps}(\T)}<R \}
\end{align*} 
satisfying $\mu_{s}(A)=0$, we have $\mu_{s}(\Phi(-t)(A))=0$ for every $t\in [-T,T]$.
\end{theorem}

We conclude this subsection with a few remarks.

\begin{remark}\rm
We note that any lowering of the global well-posedness regularity threshold, to say $\s_{0}$, for the FNLS~\eqref{FNLS} as stated in Proposition~\ref{prop: gwp} (ii), will immediately imply a corresponding improvement to Theorem~\ref{Main result} \eqref{resultals1} and Theorem~\ref{thm:localqi}. That is, we can `upgrade' from local-in-time quasi-invariance to quasi-invariance as in Theorem~\ref{Main result}, provided 
\begin{align*}
\s_{0}-\frac{1}{2}>\min\bigg( 1, \frac{11}{6}-\al \bigg).
\end{align*}
 This should be contrasted with the local-in-time quasi-invariance result in~\cite[Theorem 1.5]{PVT} for the focusing quintic NLS on $\T$. In that setting, a global flow does not exist in view of the
  presence of finite-time blow-up solutions (see for example~\cite{Ogawa}). Thus, it is impossible to remove the `local-in-time' restriction. 
\end{remark}

\begin{remark}\label{Remark: Bou, Zhis quasi result} \rm
In the works of Bourgain~\cite{Bo94} and Zhidkov~\cite{Zhid}, it was proven that for each $k\in \mathbb{N}$, the NLS has an invariant weighted Gaussian measure $\rho_k$ which is mutually absolutely continuous with the Gaussian measure $\mu_k$. Thus, the invariance of the measures $\rho_{k}$ imply the quasi-invariance of the Gaussian measures $\mu_k$ for each $k\in \mathbb{N}$. See also \cite{OTintro} for further discussion.
Theorem~\ref{Main result} extends these results to quasi-invariance of Gaussian measures $\mu_{s}$ under the flow of NLS to (not necessarily integer) regularities $s>\tfrac{5}{6}$. 
 \end{remark}

\subsection{Methodology and discussion} \label{sec:methods} 

Theorem~\ref{Main result} is an addition to a recent program initiated by Tzvetkov~\cite{tzvetkov2015quasiinvariant} based on understanding the role dispersion has on the transport properties of Gaussian measures under the flow of nonlinear Hamiltonian PDEs; see also~\cite{oh2016quasi, oh2017quasi, OhTzvetSos, OhTsutTzvet, PVT, hybrid}.
Within the context of abstract Wiener space, the classical work of Ramer~\cite{ramer} (see also \cite{Cameron}) studied the quasi-invariance of Gaussian measures under general nonlinear transformations. In the setting of $d$-dimensional nonlinear Hamiltonian PDEs, Ramer's result can be interpreted as requiring a $(d+\eps)$-degree of smoothing on the nonlinear part of the flow generated by a given PDE~\cite{tzvetkov2015quasiinvariant}. Recently, Tzvetkov~\cite{tzvetkov2015quasiinvariant} introduced a general methodology for proving quasi-invariance of Gaussian measures under the flow of nonlinear Hamiltonian PDE which does not appeal to Ramer's result. In that paper, Tzvetkov studied the generalised BBM equation and crucially made use of the explicit smoothing present on the nonlinearity to prove quasi-invariance of Gaussian measures $\mu_{s}$. However, in our case of FNLS~\eqref{FNLS}, there is no explicit smoothing on the nonlinearity. To overcome this and reveal the necessary smoothing, we employ gauge transformations and normal form reductions on~\eqref{FNLS} (see Section~\ref{section:gauge}).

Our proof of Theorem~\ref{Main result} and Theorem~\ref{thm:localqi} is split into two parts. In the first, we employ the recent argument in~\cite{PVT} (see Method 3 below) to obtain quasi-invariance for all $\al>\tfrac 12$, for some range of regularities $s$. We then apply the argument in~\cite{hybrid} (see Method 4 below) to improve upon the previous regularity restriction for $\al>\tfrac 56$.
We now go over the recent developments in the study of quasi-invariance of Gaussian measures under the flow of nonlinear Hamiltonian PDEs.
  Note that, in order to rigorously justify the methods below, it is necessary to consider a suitably truncated version of \eqref{FNLS} (see for example \eqref{Truncated equation for v} and \eqref{truncated w equation}). Furthermore, for the sake of discussion, we restrict to the one-dimensional case.

\medskip

\noi
$\bullet $ \textbf{Method 1: (`Ramer's argument')}
The first method is to directly verify the hypothesis of Ramer's result \cite{ramer} on quasi-invariance of Gaussian measures under general nonlinear transformations. In the one-dimensional context, this essentially reduces to demonstrating a $(1+\eps)$-degree of smoothing for the nonlinear part of the flow. This approach was applied in \cite{tzvetkov2015quasiinvariant, oh2016quasi, OhTsutTzvet}. 

\medskip

\noi
\underline{\textbf{`Energy methods:'}}

\medskip
\noi
$\bullet $ \textbf{Method 2: } 
Introduced by Tzvetkov~\cite{tzvetkov2015quasiinvariant}, the second method involves both nonlinear PDE techniques and stochastic analysis. We give an overview of the method here; see also \cite{OTintro}.
Let $\Psi$ be the flow of a given PDE.
 Given a measurable set $A\subset H^{s-\frac{1}{2}-\eps}(\T)$ satisfying $\mu_{s}(A)=0$, we aim to show 
\begin{align}
\mu_{s}(\Psi(t)(A))=Z_{s}^{-1}\int_{\Phi(t)(A)}e^{-\frac{1}{2}\|u\|_{H^{s}}^{2}}du=0 \,\,\, \text{for all} \,\, t\in \R,\label{push0}
\end{align}
 by obtaining a differential inequality of the form 
\begin{align}
\frac{d}{dt}\mu_{s}(\Psi(t)(A)) \leq Cp^{\beta} \{ \mu_{s}(\Psi(t)(A))\}^{1-\frac{1}{p}}, \label{ODE}
\end{align}
where $0\leq \beta \leq 1$ and $p<\infty$. Then, applying Yudovich's argument (or a variant of, see \cite{tzvetkov2015quasiinvariant, oh2017quasi}) to \eqref{ODE} implies \eqref{push0} for small times. The argument can then be iterated to give \eqref{push0} for all times. 
Thus, matters reduce to obtaining \eqref{ODE}. 
By Liouville's theorem and the bijectivity of the flow $\Psi$, we have the following `change of variables' formula (Lemma~\ref{lemma:cov}): 
\begin{align*}
\mu_{s}(\Psi(t)(A))=Z_{s}^{-1}\int_{A}e^{-\frac{1}{2}\|\Psi(t)u\|_{H^{s}}^{2}}du \,\,\, \text{for all} \,\, t\in \R.
\end{align*}
Taking a time derivative, evaluating at a fixed $t_0\in \R$ and using the group property of the flow, $\Psi(t+t_0)=\Psi(t)\Psi(t_0)$, we obtain
\begin{align}
\frac{d}{dt} \mu_{s}(\Psi(t)(A))\bigg\vert_{t=t_0} & = -\frac{1}{2} Z^{-1}_{s}\int_{\Psi(t_0)(A)} \frac{d}{dt}\Big( \|\Psi(t)(u)\|_{H^{s}}^{2} \Big) e^{-\frac{1}{2}\|\Psi(t)u\|_{H^{s}}^{2}}  \bigg\vert_{t=0}du \notag \\
& \leq C\bigg\| \frac{d}{dt}\Big( \|\Phi(t)(u)\|_{H^{s}}^{2} \Big)\bigg\vert_{t=0} \bigg\|_{L^{p}(\mu_{s})}\{ \mu_{s}(\Psi(t_0)(A))\}^{1-\frac{1}{p}}. \label{ODE1}
\end{align}
Thus, we are lead to the following energy estimate (with smoothing):
\begin{align}
\frac{d}{dt}\Big( \|\Psi(t)(u)\|_{H^{s}}^{2} \Big)\bigg\vert_{t=0} \leq C(\|u\|_{B})\norm{u}_{X}^\ta, \label{EEST1}
\end{align}
where $\ta \leq 2$.
Here, we have the freedom to choose the $X$-norm above provided it captures the regularity of the random distribution \eqref{gaussiandata} almost surely; for example, we may take $X=H^{s-\frac 12-\eps}(\T)$, the Bessel potential space $W^{s-\frac 12-\eps,\infty}(\T)$ or the Fourier-Lebesgue space $\FL^{s-\eps,\infty}(\T)$ (see~\eqref{flpspace}). 
 On the other hand,  we must choose the weaker $B$-norm so that it can be controlled in terms of conserved quantities of the given PDE. The inequality \eqref{ODE} then follows from~\eqref{ODE1}, \eqref{EEST1} and estimates on higher moments of Gaussian random random variables (see \eqref{LPgauss}).
Indeed, the reduction to time $t=0$ in the above analysis allows us to use stochastic tools on the explicit random distribution \eqref{gaussiandata}. 

For the generalised BBM equation, Tzvetkov~\cite{tzvetkov2015quasiinvariant} was able to obtain a suitable energy estimate of the form \eqref{EEST1}. Unfortunately, for general dispersive PDE such an estimate does not always hold. The key modification is to instead consider a `modified energy' $E$ of the form
\begin{align*}
E(u)=\|u\|_{H^{s}}^{2}+\text{correction terms}
\end{align*}
 and obtain the following estimate (with smoothing):
\begin{equation}\label{Tzvetkov energy estimate}
\left|\frac{d}{dt}E(\Psi(t)u) \Big\vert_{t=0} \right|\leq C(\norm{u}_{B})\norm{u}_{X}^\ta.
\end{equation}
Now, provided we show the measure $\rho_{s}$ with density
\begin{align}
d\rho_{s}=Z_{s}^{-1}e^{-E(u)}du \label{RHOs}
\end{align}
can be normalised into a probability measure, we can repeat the above argument for $\rho_{s}$ and conclude the quasi-invariance of $\rho_{s}$ under the flow $\Psi$.
Finally, we appeal to the mutual absolute continuity of $\rho_{s}$ and $\mu_{s}$ to conclude the quasi-invariance for $\mu_{s}$ under the flow $\Psi$.
To summarise, Method 2 requires two crucial ingredients: (i) a modified energy estimate of the form \eqref{Tzvetkov energy estimate} and (ii) the construction of the weighted Gaussian measure $\rho_s$ in \eqref{RHOs}. 

\medskip

\noi
$\bullet $ \textbf{Method 3:} 
Introduced by Planchon, Tzvetkov and Visciglia~\cite{PVT}, where they studied the quasi-invariance of Gaussian measures under the flow of the (super-)quintic NLS on $\T$, the third approach is similar in spirit to Method 2. 
  The fundamental feature of this method is the use of deterministic growth bounds on the  $H^{s-\frac{1}{2}-\eps}$-norm  of solutions (see Proposition~\ref{Prop: LWP growth bound}), so that the analysis can be restricted to a closed ball $B_R \subset H^{s-\frac 12-\eps}(\T)$. The benefit of this idea over Method 2 is that we require a softer energy estimate:
  \begin{equation}
\left|\frac{d}{dt}E(\Psi(t)u)\right|\leq C\big(1+\norm{\Psi(t)u}_{H^{s-\frac 12-\eps}}^{k}\big) 
\label{PVT energy estimate}
\end{equation}
for some $k\geq 0$. 
We then use \eqref{PVT energy estimate} and the growth bound on solutions to show, for any $A\subset B_R$ with $\mu_{s}(A)=0$, we have 
\begin{align*}
\frac{d}{dt}\tilde{\rho}_{s}(\Psi(t)(A))\leq C(R,T)^{k} \tilde{\rho}_{s}(\Psi(t)(A)) \quad \text{for all} \,\,\, t \in [0,T],
\end{align*}
where $\tilde{\rho}_{s}$ is the measure with density
\begin{align*}
d\tilde{\rho}_{s}=e^{-E(u)}du.
\end{align*} 
 Gronwall's inequality and soft arguments then imply 
$\tilde{\rho}_{s}(\Psi(t)(A))=0 $
and hence $\mu_{s}(\Psi(t)(A))=0
$ for every $A\subset B_R$. We then take $R\rightarrow \infty$ to obtain the quasi-invariance of $\mu_s$ under the flow $\Psi$.
Two further differences to Method 2 are: (i) there is no need to reduce to time $t=0$ to access stochastic tools and (ii) we do not need to normalise the measure $\tilde{\rho}_{s}$.
Notice that the above argument works even when the flow is only locally-in-time well-defined, which leads to a local-in-time quasi-invariance result as in Theorem~\ref{thm:localqi} (see also \cite[p. 28]{BouSurvey}).

\medskip

\noi
$\bullet $ \textbf{Method 4: } 
 This approach combines aspects of Methods 2 and 3 and was introduced by Gunaratnam, Oh, Tzvetkov and Weber~\cite{hybrid} for handling the cubic nonlinear wave equation (NLW) on $\T^3$. 
 Namely, by arguing locally within $H^{s-\frac{1}{2}-\eps}(\T)$ and returning the analysis to time $t=0$, one needs an even softer energy estimate taking the following form:
\begin{equation}\label{hybrid energy estimate}
\left|\frac{d}{dt} E(\Psi(t)v)\Big\vert_{t=0} \right|\leq C(\norm{v}_{H^{s-\frac 12 -\eps}})\norm{v}_{X}^\ta,
\end{equation}
for $\ta \leq 2$ and where the $X$-norm may be chosen as in Method 2. Analogously to Method 2, we must also construct a suitable auxiliary probability measure adapted to the modified energy.

\medskip

In the following, we survey how each of the above methods may be implemented within the context of FNLS~\eqref{FNLS}.

In practice, Method 1 is often less applicable than Methods 2, 3 and 4 as it requires a $(1+\eps)$-amount of smoothing on the nonlinear part of the flow. In our situation of FNLS~\eqref{FNLS} with $\al\geq 1$, we could demonstrate $(2\al-1)$-degrees of nonlinear smoothing, provided $s>1$. Hence, Method 1 yields quasi-invariance of Gaussian measures $\mu_{s}$ under the flow of \eqref{FNLS}, provided that $\al>1$ and $s>1$. For FNLS~\eqref{FNLS} with $\tfrac{1}{2}<\al\leq 1$, a $(1+\eps)$-degree of nonlinear smoothing is not expected~\cite{EGT} and thus we do not know at this point if Ramer's argument can be applied in this case. 
In addition, Methods 2, 3 and 4 usually give lower regularity restrictions compared to using Ramer's argument (e.g.~\cite{OhTzvetSos}). We found this to indeed be the case for FNLS~\eqref{FNLS}.
For this reason, we do not present Method 1 here. 

For application to FNLS~\eqref{FNLS}, it turns out that Method 3 gives an improved result in terms of regularity over Method 2. Indeed, in terms of the energy estimate itself, the rigidity in the choice of the $B$-norm in \eqref{Tzvetkov energy estimate} leads to far less flexibility compared to the energy estimate~\eqref{PVT energy estimate} in Method 3. 
 In the regime $\tfrac{1}{2}<\al<1$, we established an energy estimate of the form \eqref{Tzvetkov energy estimate} with $B=H^{\al}(\T)$ and $X=H^{s-\frac{1}{2}-\eps}(\T)$. Thus, energy conservation \eqref{energy} immediately places the regularity restriction $s>\al+\frac{1}{2}$, in this use of Method 2. This restriction is unnatural since it goes against our intuition that greater dispersion gives a lower regularity threshold. In this paper, we use Method 3 which allows us to remove the restrictions coming from using conservation laws and thus lower the regularity threshold.
 
We now describe our application of Method 3. 
The main goals are to establish (i) a suitable modified energy (see \eqref{modifiedenergy}) and (ii) a corresponding energy estimate of the form \eqref{PVT energy estimate} (see Proposition~\ref{Prop: Energy Estimate}). For this purpose, we apply gauge transformations to convert \eqref{FNLS} into a form more amenable to apply the normal form reductions used to define the modified energy (see Sections~\ref{section:gauge} and~\ref{SEC:QI1}). In this approach, the phase function 
\begin{align}
\phi(\cj n)=|n_1|^{2\al}-|n_2|^{2\al}+|n_3|^{2\al}-|n|^{2\al} \label{p1}
\end{align}
naturally arises as the source of dispersion. In order to exploit this for a smoothing benefit, which is required to achieve (ii), we crucially rely on the following lower bound: for $\al>\tfrac 12$, we have
\begin{align*}
\left|\phi(\overline{n})\right|
& \gtrsim |n-n_1||n-n_3|n_{\max}^{2\alpha-2} \qquad \text{when} \qquad n=n_1-n_2+n_3.
\end{align*}
Here, $n_{\max}:=\max( |n_1|,|n_2|,|n_3|,|n|)$.
This lower bound first appeared in the setting $\frac{1}{2}<\al\leq 1$ in \cite{demirbas2013existence}. 
With minor modifications, its proof extends easily to the case $\al>1$; see Lemma \ref{lemma: phase lower bound} and Appendix~\ref{app:lowerbd}. It can be viewed as a replacement of the explicit factorisations available for the phase function \eqref{p1} of NLS 
\begin{align*}
\phi(\cj n)=n_1^{2}-n_2^2+n_3^2-n^2=-2(n-n_1)(n-n_3)  \qquad \text{when }\, n=n_1-n_2+n_3,
\end{align*}
and 4NLS (see \cite[Lemma 3.1]{oh2016quasi}). Following the argument in \cite{PVT}, we obtain quasi-invariance of Gaussian measures $\mu_{s}$ under the flow of \eqref{FNLS} for
\begin{align}
\textup{(i)} \,\, s>1, \,\, \text{when  } \frac 12 <\al\leq 1, \,\,\,  \text{and} \,\,\, \textup{(ii)} \,\, s>s_{\al},\,\, \text{for some } s_{\al}\leq 1, \,\, \text{when  } \al>1.
\label{M3REG}
\end{align}
See \eqref{highalpha} for a precise statement of $s_{\al}$.

Our next goal is to attempt to lower the regularity restriction from Method 3 by using Method 4.
This requires us to construct a suitable weighted Gaussian measure (Subsection~\ref{subsec:probmeas}) and establish an effective energy estimate of the form~\eqref{hybrid energy estimate}. In establishing the energy estimate, we have some freedom in the choice of the $X$-norm. One choice is the H\"{o}lder-Besov norm as used in~\cite{hybrid}. Since we work intimately on the Fourier side, we use the {\it Fourier-Lebesgue} $X=\mathcal{F}L^{\s, \, \infty}(\T)$-norm for $\s<s$. Here, given $q\geq 1$ and $s\in \R$, the Fourier-Lebesgue $\FL^{s,q}(\T)$-norm is defined by: 
 \begin{align} \label{flpspace}
\| f\|_{\FL^{s,q}(\T)}:=\| \jb{n}^{s}\ft f(n)\|_{\l^{q}(\Z)}.
\end{align}
It is easy to check that the random distribution in \eqref{gaussiandata} belongs almost surely to $\FL^{\s , \, \infty}(\T)$ for any $\s<s$ (Lemma~\ref{lemma:flpdeviation}). Moreover, H\"{o}lder's inequality implies the embedding
\begin{align}
H^{s-\frac 12-\eps}(\T) \supset \FL^{\s,\infty}(\T) \label{FLembed}
\end{align} 
for $\s$ sufficiently close to $s$.
This fact allows us to further relax the energy estimate we obtained in Method 3 (see Proposition~\ref{prop:weakenergyest}). 
We then follow the argument in \cite{hybrid} to conclude quasi-invariance of Gaussian measures $\mu_{s}$ under the flow of FNLS~\eqref{FNLS} for the following regularities:
\begin{align}
\begin{split}
\textup{(i)}& \,\,\max\bigg( \frac{2}{3},\frac{11}{6}-\al \bigg)<s\leq 1,\,\, \text{when  } \al\geq 1 \,\,\, \text{and}\\
\textup{(ii)}& \,\, \frac{11}{6}-\al <s\leq 1 , \,\, \text{when  } \frac 56 <\al< 1. 
\label{M4REG}
\end{split}
\end{align}
Notice that in \eqref{M4REG}, we improve the regularity restriction \eqref{M3REG} we obtained using Method 3 only when $\al>\tfrac 56$. The reason for this is that our use of the stronger $\FL^{\s,\infty}$-norm in the energy estimate for Method 4 (Proposition~\ref{prop:weakenergyest}) yields a regularity gain over the energy estimate in Method 3 (Proposition~\ref{Prop: Energy Estimate}) provided $\al>\tfrac 56$. 
Furthermore, the upper bound $s\leq 1$ in \eqref{M4REG} is necessary for our construction of the weighted Gaussian measure.

\begin{remark}\rm 
When $\al \geq \tfrac{7}{6}$, the regularity restriction in Theorem~\ref{Main result} \eqref{resultals} achieves the largest range of $s>\frac{2}{3}$. In particular, when $\al=2$, this improves upon the result in \cite{oh2016quasi} of $s>\tfrac{3}{4}$. However, as remarked in \cite{OhTsutTzvet}, this same result of $s>\tfrac{2}{3}$ for 4NLS could be obtained by using Method 1 and an additional novel gauge transformation introduced in that same paper. 
For FNLS~\eqref{FNLS} with $\al>1$ (and large enough), we expect the optimal result $s>\tfrac 12$ could be obtained by using a finer modified energy arising from an infinite sequence of normal form reductions. See \cite{OhTzvetSos} where this approach led to the optimal result for 4NLS.

Recently, Oh, Tzvetkov and Wang~\cite{OTW} established the invariance of the (Gibbs-type) measure $\mu_{0}$ under the flow of the (renormalised) 4NLS; namely \eqref{Equation for v} for $\al=2$. Their analysis seems to also extend to the (renormalised) FNLS~\eqref{Equation  for v} for some $1<\al<2$. Thus in the presence of naturally associated invariant measures, quasi-invariance may persist for some regularities below what \eqref{soptimal} suggests.
\end{remark}

\begin{remark}\rm
In Theorem~\ref{Main result} and Theorem~\ref{thm:localqi}, we studied the quasi-invariance of Gaussian measures under the flow of FNLS~\eqref{FNLS} for $\al>\tfrac 12$.
A natural question would be to study the transport property of Gaussian measures under the flow of FNLS~\eqref{FNLS} with $\al=\tfrac 12$. We recall this corresponds to the non-dispersive half-wave equation
  \begin{align}
i\dt u + |\dx|u= |u|^{2}u. \label{halfwave}
\end{align}
 In this case, we expect the Gaussian measures to not be quasi-invariant under the flow of \eqref{halfwave}.   
   To give some credence to this, we observe that it was shown in~\cite{OhTzvetSos} that Gaussian measures $\mu_{s}$ are not quasi-invariant under the flow of the \emph{dispersionless} equation
   \begin{align}
i\dt u=|u|^{2}u. \label{displess}
\end{align}
  Given a solution $u$ to \eqref{displess}, the change of variables 
\begin{equation*}
u(t,x)\longmapsto u(t,x-t)
\end{equation*} 
 implies
$\mu_s$ is not quasi-invariant under the flow of 
\begin{equation*}
i\partial_t u+i\partial_xu=|u|^2u \label{transporteq}
\end{equation*}
which closely resembles \eqref{halfwave}. 
The proof of the non quasi-invariance under \eqref{displess} in~\cite{OhTzvetSos} heavily makes use of the explicit solution formula  
\begin{equation*}
u(t,x)=e^{-it|u(0,x)|^2}u(0,x).
\end{equation*}
Unfortunately, there is no such solution formula for the half-wave equation~\eqref{halfwave}. Hence, at this point, we do not know how to conclude non quasi-invariance of Gaussian measures under the flow of \eqref{halfwave}. It would also be of interest to consider the transport properties of Gaussian measures under the flow of other dispersionless PDE, such as the cubic Sz\"{e}go equation:
\begin{align*}
i\dt u= \P_{\geq 0}(|u|^2 u),
\end{align*}
where $\P_{\geq 0}$ is the projection onto non-negative frequencies $\{ n\, : n \geq 0\}$. See~\cite{Gerard} and the references therein for more on the cubic Sz\"{e}go equation.
\end{remark}

\begin{remark}\rm
It would be of interest to study how our approach in this paper may extend to higher order nonlinearities, say, for the quintic nonlinearity $|u|^{4}u$. 
For instance, the relevant phase function is now 
\begin{align}
|n_1|^{2\al}-|n_2|^{2\al}+|n_3|^{2\al}-|n_4|^{2\al}+|n_5|^{2\al}-|n|^{2\al}, 
\label{Quintic1}
\end{align}
which is restricted to the hyperplane $n_1-n_2+n_3-n_4+n_5=n$. When $\al=1$, there is no factorisation for \eqref{Quintic1}. Therefore, an appropriate analogue of Lemma~\ref{lemma: phase lower bound} is not clear (and likewise for Lemma~\ref{Lemma: DMVT application}). However, even if such results were proved, the formulation of the modified energies and the appropriate nonlinear estimates would still have to be verified. We note that the method in \cite{PVT} introduces a modified energy functional which is not derivable from differentiation by parts, and hence their analysis is not based on factorisations of the phase function \eqref{Quintic1}.
\end{remark}

\section{Preliminary estimates}
In this section we record some elementary estimates that will be useful in the coming analysis. 
The first result we need is the double mean value theorem (DMVT) from~\cite[Lemma 2.3]{CKSTT}.

\begin{lemma}[DMVT]\label{DMVT}
Let $\xi,\eta,\lambda\in \R$ and $f\in C^2(\R)$. Then, we have
\begin{equation*}
f(\xi+\eta+\lambda)-f(\xi+\eta)-f(\xi+\lambda)+f(\xi)=\lambda\eta\int_0^1\int_0^1f''(\xi+s\lambda+t\eta)\,dsdt.
\end{equation*}
\end{lemma}

We have the following consequence of DMVT:

\begin{lemma}\label{Lemma: DMVT application}
Fix $s> 1$ and let $n_1,n_2,n_3,n\in \Z$ be such that $n=n_1-n_2+n_3$. Then, we have
\begin{equation*}
\left |\jb{n_1}^{2s}-\jb{n_2}^{2s}+\jb{n_3}^{2s}-\jb{n}^{2s}  \right|\lesssim |n-n_1| |n-n_3|\jb{n_{\max}}^{2s-2},
\end{equation*}
where $n_{\max} = \max(|n_1|,|n_2|,|n_3|,|n|)$ and the implicit constant depends only on $s$.
\end{lemma}
\begin{proof}
This is a simple application of DMVT upon setting $n_1=\xi+\eta+\lambda$, $n_2=\xi+\eta$, $n=\xi+\lambda$ and $n_3=\xi$.
\end{proof}

The next lemma states a crucial lower bound on the phase function $\phi(\cj n)$ of \eqref{p1} which we use repeatedly throughout. It was proved for the case $\tfrac 12 < \al \leq 1$ in \cite{demirbas2013existence}.  Their proof easily extends to the case $\al>1$; see Appendix~\ref{app:lowerbd}.
\begin{lemma}\label{lemma: phase lower bound}
Fix $\alpha > \tfrac{1}{2}$ and let $n_1,n_2,n_3,n\in \Z$ be such that $n=n_1-n_2+n_3$. Then, we have
\begin{align*}
\left|\phi(\overline{n})\right|&\gtrsim |n-n_1||n-n_3|\left(  |n-n_1|+|n-n_3|+|n|\right)^{2\alpha-2} \\
& \gtrsim |n-n_1||n-n_3|n_{\textup{max}}^{2\alpha-2}
\end{align*}
where $n_{\max} = \max(|n_1|,|n_2|,|n_3|,|n|)$ and the implicit constant depends only on $\alpha$.
\end{lemma}

We next state a useful summing estimate, a proof of which can be found in, for example, \cite[Lemma 4.2]{GTV}.

\begin{lemma}\label{lemma:sumestimate} If $\beta\geq \gamma \geq 0$ and $\beta+\gamma >1$, then we have 
\begin{align*}
\sum_{n}\frac{1}{\jb{n-k_{1}}^{\beta}\jb{n-k_{2}}^{\gamma}}&\lesssim \frac{\varphi_{\beta}(k_{1}-k_{2})}{\jb{k_{1}-k_{2}}^{\gamma}}, \\
\int_{\R} \frac{1}{\jb{x-k_{1}}^{\beta}\jb{x-k_{2}}^{\gamma}}dx &\les \frac{\varphi_{\beta}(k_{1}-k_{2})}{\jb{k_{1}-k_{2}}^{\gamma}},
\end{align*}
  where 
  \begin{align*}
\varphi_{\beta}(k):=\sum_{1\leq |n|\leq |k|}\frac{1}{|n|^{\beta}}\sim 
\begin{cases} 
1, \quad  &\textup{if } \,\,\beta>1, \\ 
\log(1+\jb{k}), \quad& \textup{if } \,\, \beta=1, \\
 \jb{k}^{1-\beta}, \quad & \textup{if } \,\, \beta<1.
\end{cases}
\end{align*}
 \end{lemma}

Finally, we will require the following fact from elementary number theory~\cite{divcnt}:
Given $n\in \mathbb{N}$ and any $\dl>0$, there exists a constant $C_{\dl}>0$ such that the number of divisors $d(n)$ of $n$ satisfies 
\begin{align}
d(n) \leq C_{\dl} n^{\dl}. \label{divisorcount}
\end{align}

\section{Reformulation of FNLS}\label{section:gauge}
In this section, we reformulate FNLS~\eqref{FNLS} into a more amenable form for the normal form reductions in the next section. Given $t\in \R$, we consider the gauge transform $\mathcal{G}_t$ on $L^2(\T)$ defined by 
\begin{equation*}
\mathcal{G}_t[f]=e^{2it\fint |f|^2 dx}f, \
\end{equation*}
where $\fint_{\T}f(x)\,dx :=\frac{1}{2\pi}\int_{\T}f(x)\,dx$. Furthermore, given $u\in C(\R;L^2(\T))$, we define $\mathcal{G}$ by 
\begin{equation*}
\mathcal{G}[u](t):=\mathcal{G}_t[u(t)].
\end{equation*}
It is easy to check that $\mathcal{G}$ is invertible with inverse 
\begin{equation*}
\mathcal{G}^{-1}[u](t)=\mathcal{G}_{-t}[u(t)].
\end{equation*}
Now, let $u\in C(\R;L^{2}(\T))$ be a solution to (\ref{FNLS}) and define $v$ by 
\begin{equation*}
v(t)=\mathcal{G}[u](t).
\end{equation*}
Then it follows from mass conservation~\eqref{mass} that $v$ satisfies 
\begin{equation}\label{Equation for v}
i\partial_tv+(-\partial_x^2)^{\alpha}v=\left(|v|^2-2\fint_{\T}|v|^2\,dx \right)v.
\end{equation}
Namely, $v$ satisfies \eqref{FNLS} but with a more favourable nonlinearity.
We define $\Phi(t): u_0\in H^{\s}(\T) \mapsto v(t)\in H^{\s}(\T)$ to be the solution map of~\eqref{Equation for v} (when it exists) at time $t$.

In order to make the following calculations secure, we consider the following truncated equation:
\begin{equation}\label{Truncated equation for v}
i\partial_tv+(-\partial_x^2)^{\alpha}v=\mathbf{P}_{\leq N}\left[ \left(|\mathbf{P}_{\leq N}v|^2-2\fint_{\T}|\mathbf{P}_{\leq N}v|^2\,dx \right)\mathbf{P}_{\leq N}v\right].
\end{equation}
Here, $\P_{\leq N}$ is the projection onto frequencies $\{ n\, : \, |n|\leq N\}$ for $N\in \mathbb{N}$.
We let $\Phi_N(t)$ denote the solution map of \eqref{Truncated equation for v} at time $t$ (when it exists).

To exploit the dispersive nature of \eqref{Equation for v}, we will need another gauge transform. We define the {\it interaction representation} of $v$ as 
\begin{equation}
w(t)=S(-t)v(t),   \label{interaction representation}
\end{equation}
where $S(t)=e^{it(-\partial_x^2)^{\alpha}}$. On the Fourier side, we have\footnote{For clarity, we will sometimes write $\ft f(n)$ as $\ft {f}_{n}$.}
\begin{equation*}
\ft w_n(t)=e^{-it|n|^{2\alpha}}\ft v_n(t).
\end{equation*}
Then, the equation \eqref{Equation for v} becomes the following equation for the Fourier coefficients $\{\ft w_n\}_{n\in \Z}$:
\begin{align} \label{equation for w}
\begin{split}
\partial_t \ft w_n &=-i\sum_{\Gamma(n)}e^{it\phi(\bar{n})}\ft w_{n_1}\overline{\ft w_{n_2}}\ft w_{n_3}+i|\ft w_n|^2 \ft w_n,
\end{split}
\end{align}
where the phase function $\phi(\overline{n})$ and the plane $\Gamma(n)$ are given by
\begin{equation*}
\phi(\overline{n})=\phi(n_1,n_2,n_3,n)=|n_1|^{2\alpha}-|n_2|^{2\alpha}+|n_3|^{2\alpha}-|n|^{2\alpha}
\end{equation*}
and
\begin{equation}
\Gamma(n)=\{(n_1,n_2,n_3)\in \Z^3:n=n_1-n_2+n_3 \textnormal{ and }n_1,n_3\neq n \}. \label{nrplane}
\end{equation}
Similarly, the truncated equation \eqref{Truncated equation for v} becomes the following equation for the Fourier coefficients $\{\ft w_n\}_{n\in \Z}$:
\begin{align} \label{truncated w equation}
\begin{split}
\partial_t \ft w_n=\mathbf{1}_{|n|\leq N}\left[-i\sum_{\Gamma_{N}(n)}e^{it\phi(\bar{n})}\ft w_{n_1}\overline{  \ft w_{n_2}}\ft w_{n_3}+i|\ft w_n|^2 \ft w_n\right],
\end{split}
\end{align}
where the plane $\Gamma_{N}(n)$ is given by
\begin{equation*}
\Gamma_N(n)=\Gamma(n)\cap\{(n_1,n_2,n_3):|n_j|\leq N, \,\, j=1,2,3 \}.
\end{equation*}
From now on, for ease of notation, we will typically ignore the `hats' on the Fourier coefficients. 

The following lemma shows that it suffices to prove the quasi-invariance of $\mu_s$ under the flow of \eqref{Equation for v}.
\begin{lemma}\label{Gauge transform quasi}
The following is true:
\begin{enumerate}[\normalfont (i)]
\setlength\itemsep{0.3em}
\item Let $s>\frac{1}{2}$. Then, for any $t\in \R$, the Gaussian measure $\mu_s$ is invariant under the map $\mathcal{G}_t$.
\item Let $(X,\mu)$ be a measure space and suppose that $T_1$ and $T_2$ are maps from $X$ to itself such that $\mu$ is quasi-invariant under $T_1$ and is quasi-invariant under $T_2$. Then, $\mu$ is quasi-invariant under the composition $T_1\circ T_2$.
\end{enumerate}
\end{lemma}

For a proof of these, see \cite[Lemmas 4.4 and 4.5]{oh2016quasi}. We note in the particular case $s=1$, (i) follows from the results in \cite{NahmodStaff}.

\section{Quasi-invariance for $\al>\frac 12$}\label{SEC:QI1}

In this section, we present part of the proof of Theorem~\ref{Main result} and Theorem~\ref{thm:localqi} by applying the argument in \cite{PVT} (Method 3). Namely, we establish quasi-invariance of Gaussian measures $\mu_{s}$ under the flow of FNLS~\eqref{FNLS} for regularities $s$ given in \eqref{M3REG}. We begin in Subsection \ref{subsec:energy} by deriving a suitable modified energy and obtaining the key energy estimate of the form \eqref{PVT energy estimate}. Then, in Subsection~\ref{subsec:concludeQI1}, we use this energy estimate to conclude the quasi-invariance of $\mu_{s}$.

\subsection{Energy estimate} \label{subsec:energy}

Given a smooth solution $v$ to \eqref{Equation for v}, let $w$ be as in \eqref{interaction representation}. Then from \eqref{equation for w}, we have
\begin{align}\label{Hs norm time derivative}
\begin{split}
\frac{d}{dt}\lVert v(t)\rVert_{H^s}^2&=\frac{d}{dt}\lVert w(t)\rVert_{H^s}^2 \\
&=-2\, \mbox{Re}\, i \sum\limits_{n\in \Z}\sum_{\Gamma(n)}\langle n\rangle^{2s}e^{it\phi(\overline{n})}w_{n_1}\overline{w_{n_2}}w_{n_3}\overline{w_{n}}\\
&=\frac{1}{2}\, \textnormal{Re} \, i\sum\limits_{\G(\cj n)}\psi_s(\overline{n})e^{it\phi(\overline{n})}w_{n_1}\overline{w_{n_2}}w_{n_3}\overline{w_{n}},
\end{split}
\end{align}
where $\overline{n}=(n_1,n_2,n_3,n)$,
\begin{equation*}
\Gamma(\cj n) := \{(n_1,n_2,n_3,n)\in \Z^4: n_1-n_2+n_3=n \textnormal{ and } n_1, n_3\neq n\} 
\end{equation*}
and
\begin{align}
\psi_s(\overline{n}) &=\langle n_1\rangle^{2s}-\langle n_2\rangle^{2s}+\langle n_3\rangle^{2s}-\langle n\rangle^{2s}. \label{derivs}
\end{align}
The second equality in (\ref{Hs norm time derivative}) follows by a symmetrisation argument. Indeed, a relabelling of the sum implies
\begin{equation*}
\textnormal{Re}\, i\sum\limits_{\G(\cj n)}\jb{n}^{2s} e^{it\phi(\overline{n})}w_{n_1}\overline{w_{n_2}}w_{n_3}\overline{w_{n}}=\textnormal{Re}\, i\sum\limits_{\G(\cj n)}\langle n_2\rangle^{2s} e^{i\phi(\overline{n})t}w_{n_1}\overline{w_{n_2}}w_{n_3}\overline{w_{n}}.
\end{equation*}
Using the fact that $\Re i a=-\Re i\overline{a}$ for all $a\in\C$, a relabelling also shows
\begin{equation*}
\textnormal{Re}\, i\sum\limits_{\G(\cj n)}\langle n\rangle^{2s} e^{it\phi(\overline{n})}w_{n_1}\overline{w_{n_2}}w_{n_3}\overline{w_{n}}=-\textnormal{Re}\, i\sum\limits_{\G(\cj n)}\langle n_1\rangle^{2s} e^{it\phi(\overline{n})}w_{n_1}\overline{w_{n_2}}w_{n_3}\overline{w_{n}}.
\end{equation*}
This symmetrization puts us in a position to apply Lemma \ref{Lemma: DMVT application} later. 
Writing 
\begin{align*}
\frac{d}{dt} \bigg(  \frac{e^{it\phi(\cj n)}}{i\phi(\cj n)} \bigg)=e^{it\phi(\cj n)},
\end{align*}
and applying the product rule in reverse, 
\eqref{Hs norm time derivative} implies
\begin{align}\label{derivative calc}
\begin{split}
\frac{d}{dt}\lVert w(t)\rVert_{H^s}^2=& \frac{1}{2}\,\mbox{Re}\frac{d}{dt} \left[ \sum\limits_{\G(\cj n)}\frac{\psi_s(\overline{n})}{\phi(\overline{n}) }e^{it\phi(\overline{n})} w_{n_1}\overline{w_{n_2}}w_{n_3}\overline{w_{n}}   \right]\\
&-\frac{1}{2}\,\mbox{Re}\sum\limits_{\G(\cj n)}\frac{\psi_s(\overline{n})}{\phi(\overline{n}) }e^{it\phi(\overline{n})}\partial_t(w_{n_1}\overline{w_{n_2}}w_{n_3}\overline{w_{n}}). 
\end{split}
\end{align}
We now define the modified energy:
\begin{align*}
E_{s,t}(z)& =\lVert z\rVert_{H^s}^2 -\frac{1}{2}\, \mbox{Re}\sum\limits_{\G(\cj n)}\frac{\psi_s(\overline{n})}{\phi(\overline{n}) }e^{it\phi(\overline{n})}z_{n_1}\overline{z_{n_2}}z_{n_3}\overline{z_{n}}\\
&=: \norm{z}_{H^s}^2+R_{s,t}(z).
\end{align*}
Then, it follows from \eqref{derivative calc} that for any solution $w$ to \eqref{equation for w}, we have 
\begin{equation}
\frac{d}{dt}E_{s,t}(w)=-\frac{1}{2}\,\mbox{Re}\sum\limits_{\G(\cj n)}\frac{\psi_s(\overline{n})}{\phi(\overline{n}) }e^{it\phi(\overline{n})}\partial_t(w_{n_1}\overline{w_{n_2}}w_{n_3}\overline{w_{n}}). \label{modenergytderiv}
\end{equation}

At first glance it seems like the modified energy \eqref{modifiedenergy} is non-autonomous in time. However, this time dependence is only superficial. Writing the modified energy in terms of $y:=S(t)z$, we have
\begin{align}
E_s(y)&:=E_{s,t}(S(-t)y)=\lVert y\rVert_{H^s}^2+R_s(y), 
 \label{modifiedenergy}
\end{align}
where
\begin{align}
R_{s}(y):=- \frac{1}{2}\, \mbox{Re}\sum\limits_{\G(\cj n)}\frac{\psi_s(\overline{n})}{\phi(\overline{n})}y_{n_1}\overline{y_{n_2}}y_{n_3}\overline{y_{n}}. \label{RS}
\end{align}
Now, the nonlinear functionals $E_s$ and $R_s$ are clearly autonomous in time.

We now state the following key energy estimate which is of the form \eqref{PVT energy estimate}.

\begin{proposition}\label{Prop: Energy Estimate}
Let $(s,\al)$ belong to one of the following regions:
 \begin{align}
 \begin{split}
\textup{(i)} \hphantom{,xxXXXXXXXXXXX} s>1, &\qquad \textup{when}\qquad \al>\frac{1}{2}, \\
\textup{(ii)} \qquad \max\bigg( \frac 23, \frac{25}{12}-\al\bigg)<s\leq 1, &\qquad \textup{when}\qquad  \al\geq \frac{5}{4},\\
\textup{(iii)} \hphantom{,xXXXXXXX} \frac{3-\al}{2} <s \leq 1, &\qquad   \textup{when}\qquad  1<\al<\frac{5}{4}.
\end{split} \label{highalpha}
\end{align}
 Then, for sufficiently small $\eps>0$, there exists $C>0$ such that 
\begin{equation}
\left|\frac{d}{dt}E_s(\mathbf{P}_{\leq N}v(t)) \right|\leq C\norm{v(t)}_{H^{s-\frac{1}{2}-\eps}}^6, \label{energyest}
\end{equation}
for all $N\in \mathbb{N}$ and any solution $v$ to \eqref{Truncated equation for v}, uniformly in $t\in\R$.
\end{proposition}

\begin{proof}
Using \eqref{modifiedenergy} and the unitarity of $S(t)$ on $H^{s-\frac 12-\eps}(\T)$, it suffices to prove, that for small $\eps>0$, there exists $C>0$ such that
\begin{equation}
\left|\frac{d}{dt}E_{s,t}(\mathbf{P}_{\leq N}w(t)) \right|\leq C\norm{w(t)}_{H^{s-\frac{1}{2}-\eps}}^6 \label{energyw}
\end{equation}
for all $N\in \mathbb{N}$ and any solution $w$ to \eqref{truncated w equation}, uniformly in $t\in \R$.

Using \eqref{modifiedenergy}, \eqref{truncated w equation} and the symmetry between $n_1$ and $n_3$ and between $n_2$ and $n$ in (the appropriate version of) \eqref{modenergytderiv}, we have
\begin{equation}
\frac{d}{dt}E_{s,t}(\P_{\leq N}w) =\mathcal{N}_1(w)+\mathcal{R}_1(w)+\mathcal{N}_2(w)+\mathcal{R}_2(w), \label{derivenergy}
\end{equation}
where
\begin{align}
\mathcal{N}_1(w)&:= -\mbox{Re}\, i\sum\limits_{\G_{N}(\cj n)}\frac{\psi_s(\overline{n})}{\phi(\overline{n}) }e^{it\phi(\overline{n})} \left (\sum\limits_{\Gamma_N(n_1)}e^{it\phi(\overline{m},n_1)}w_{m_1}\overline{w_{m_2}}w_{m_3} \right ) \overline{w_{n_2}}w_{n_3}\overline{w_{n}},  \label{N1}  \\
\mathcal{R}_1(w)&:= -\mbox{Re}\, i\sum\limits_{\G_{N}(\cj n)}\frac{\psi_s(\overline{n})}{\phi(\overline{n}) }e^{it\phi(\overline{n})} |w_{n_1}|^2w_{n_1} \overline{w_{n_2}}w_{n_3}\overline{w_{n}},  \label{R1}  \\
\mathcal{N}_2(w)&:=  -\mbox{Re}\, i\sum\limits_{\G_{N}(\cj n)}\frac{\psi_s(\overline{n})}{\phi(\overline{n}) }e^{it\phi(\overline{n})} w_{n_1}\left (\sum\limits_{\Gamma_N(n_2)}e^{-it\phi(\overline{m},n_2)}\overline{w_{m_1}}w_{m_2}\overline{w_{m_3}} \right )w_{n_3}\overline{w_{n}}, \label{N2} \\
\mathcal{R}_2(w)&:=- \mbox{Re} \, i\sum\limits_{\G_{N}(\cj n)}\frac{\psi_s(\overline{n})}{\phi(\overline{n}) }e^{it\phi(\overline{n})} w_{n_1} |w_{n_2}|^2\overline{w_{n_2}}w_{n_3}\overline{w_{n}}. \label{R2}
\end{align}
Here,
\begin{align*}
\phi(\cj m,n_1):=|m_1|^{2\al}-|m_2|^{2\al}+|m_3|^{2\al}-|n_1|^{2\al}
\end{align*}
and
\begin{align*}
\G_{N}(\cj n):=\{(n_1,n_2,n_3,n)\in \G(\cj n): |n_j|, |n| \leq N, \,\, j=1,2,3 \}.
\end{align*}
From now on, we will simply write $\G(\cj n)$ instead of $\G_{N}(\cj n)$ as $N$ plays no further role. 
In the following, we heavily make use of the Fourier lattice property of $H^{s}(\T)$; namely, that the $H^{s}$-norm depends only on the absolute value of the Fourier coefficients.
In particular, we assume all the Fourier coefficients $w_n$  are real and non-negative. Moreover, as we never make use of the oscillatory factors 
such as $e^{it\phi(\cj n)}$, we will neglect explicitly writing them.

Consider the first scenario (i). 
For $s>1$, Lemma \ref{lemma: phase lower bound} and Lemma \ref{Lemma: DMVT application} imply we have
\begin{equation}\label{phase lb and DMVT ap}
\frac{|\psi_s(\overline{n})|}{|\phi(\overline{n}) |}\leq \jb{n_{\max}}^{2s-2\alpha}.
\end{equation}
We first estimate $\mathcal{N}_1(w)$ by decomposing the sum into two cases depending on which frequency attains $n_{\max}$. 

\medskip
\noi
$\bullet$ \textbf{Case 1:} $n_\textnormal{max}=|n_1|$

From the conditions $n=n_1-n_2+n_3$ and $n_1=m_1-m_2+m_3$ we have 
$\max(|n_2|,|n_3|,|n|) \gtrsim |n_1|$ and $\max_{j=1,2,3}|m_j|\gtrsim |n_1|$, respectively. We assume  $|n_2|\gtrsim |n_1|$ and $|m_1|\gtrsim |n_1|$ as the other cases are similar. Hence, we have 
\begin{align}
\jb{n_{\max}}^{2s-2\al}\les \jb{m_1}^{s-\frac{1}{2}-\eps}\jb{n_2}^{s-\frac 12-\eps}\jb{n_{\max}}^{-2\al+1+2\eps}\les  \jb{m_1}^{s-\frac{1}{2}-\eps}\jb{n_2}^{s-\frac 12-\eps}. \label{nmaxthing}
\end{align}
Using \eqref{phase lb and DMVT ap}, \eqref{nmaxthing} and Young's inequality for the convolution of sequences we have 
\begin{align*}
|\mathcal{N}_1(w)| &\lesssim \sum\limits_{\G(\cj n)} \left (\sum\limits_{\G(n_1)}\langle m_1  \rangle^{s-\frac{1}{2}-\eps}w_{m_1}w_{m_2}w_{m_3} \right ) \langle n_2 \rangle^{s-\frac{1}{2}-\eps}w_{n_2}w_{n_3}w_{n}\\
& \lesssim \left\lVert\sum\limits_{\Gamma(n_1)}\langle m_1  \rangle^{s-\frac{1}{2}-\eps}w_{m_1}w_{m_2}w_{m_3}  \right\rVert_{\ell^2_{n_1}}\left\lVert\langle n \rangle^{s-\frac{1}{2}-\eps}w_{n}\right\rVert_{\ell^2_{n}}\left\lVert w_n\right\rVert^2_{\ell^1_{n}}.
\end{align*}
 A further application of Young's inequality gives
\begin{equation*}
\left\lVert\sum\limits_{\Gamma(n_1)}\langle m_1  \rangle^{s-\frac{1}{2}-\eps}w_{m_1}w_{m_2}w_{m_3} \right\rVert_{\ell^2_{n_1}}\lesssim \left\lVert \langle n\rangle^{s-\frac{1}{2}-\eps} w_n  \right\rVert_{\ell^2_n} \left\lVert w_n \right\rVert_{\ell^1_n}^2.
\end{equation*}
By H\"{o}lder's inequality and choosing $\eps$ small enough so that $\frac{1}{2}+\eps \leq s-\frac{1}{2}-\eps$ , we have
\begin{equation*}
\norm{w_n}_{\l^1_n}\lesssim \norm{w}_{H^{\frac{1}{2}+\eps}}\lesssim \norm{w}_{H^{s-\frac{1}{2}-\eps}}.
\end{equation*}
Putting this together we get
\begin{equation*}
|\mathcal{N}_1(w)|\lesssim \lVert v\rVert_{H^{s-\frac{1}{2}-\eps}}^6,
\end{equation*}
which is the desired estimate for $\mathcal{N}_1(w)$. 

\medskip
\noi 
$\bullet$ \textbf{Case 2:}  $n_\textnormal{max}\in \{ |n_2|, |n_3|, |n| \}$ 

It suffices to assume $n_{\max}=|n_2|$ as the remaining cases follow analogously as below.
Similar to Case 1, we have $\max(|n_1|,|n_3|,|n|) \gtrsim |n_2| $. If $|n_1|\gtrsim |n_2|$, we proceed in exactly the same way as Case 1. If instead $|n_3|\gtrsim |n_2|$ or $|n|\gtrsim |n_2|$, say the former as both subcases are similar, we use Young's inequality 
\begin{align*}
|\mathcal{N}_1(w)| &\lesssim \sum\limits_{\G(\cj n)} \left (\sum\limits_{\Gamma(n_1)}w_{m_1}w_{m_2}w_{m_3} \right ) \langle n_2 \rangle^{s-\frac{1}{2}-\eps}w_{n_2}\langle n_3  \rangle^{s-\frac{1}{2}-\eps}w_{n_3}w_{n}\\
& \lesssim \left\lVert\sum\limits_{\Gamma(n_1)}w_{m_1}w_{m_2}w_{m_3}  \right\rVert_{\ell^1_{n_1}}\left\lVert\langle n \rangle^{s-\frac{1}{2}-\eps}w_{n}\right\rVert_{\ell^2_{n}}^2\left\lVert w_n\right\rVert_{\ell^2_{n}}. 
\end{align*}
A further application of Young's inequality and H\"{o}lder's inequality gives
\begin{equation*}
\left\lVert\sum\limits_{\Gamma(n_1)}w_{m_1}w_{m_2}w_{m_3}  \right\rVert_{\ell^1_{n_1}}\lesssim \norm{w_n}_{\ell^1_n}^{3}\lesssim \norm{w}_{H^{\frac{1}{2}+\eps}}^3.
\end{equation*}
This completes the case $n_{\max}=|n_2|$ and hence the estimate for $\N_{1}(w)$.
The estimate for $\mathcal{N}_2(w)$ follows from similar arguments. Now we estimate $\mathcal{R}_1(w)$. 
\medskip

\noi
$\bullet$ \textbf{Case 1:} $n_\textnormal{max}=|n_1|$

As before, $\max(|n_2|,|n_3|,|n|) \gtrsim |n_1|$. It suffices to assume $|n_2|\gtrsim |n_1|$, as the subcases $|n_3|\gtrsim |n_1|$ and $|n|\gtrsim |n_1|$ are similar. We have,
\begin{align*}
|\mathcal{R}_1(w)| 
&\lesssim \sum_{\G(\cj n)} \langle n_1\rangle^{s-\frac{1}{2}-\eps}w_{n_1}^3 \langle n_2 \rangle^{s-\frac{1}{2}-\eps}w_{n_2}w_{n_3}w_{n}\\
&\lesssim \lVert \langle n\rangle^{s-\frac{1}{2}-\eps} w_n^3 \rVert_{\ell_n^2} \lVert \langle n \rangle^{s-\frac{1}{2}-\eps} w_n\rVert_{\ell^2_n} \lVert w_{n}\rVert_{\ell^1_n}^2.
\end{align*}
Using H\"olders inequality and then the embedding $\ell^\infty_{n}\subset \ell^2_{n}$ we have
\begin{equation*}
\lVert \langle n\rangle^{s-\frac{1}{2}-\eps} w_n^3\rVert_{\ell_n^2}\lesssim \norm{\langle n\rangle^{s-\frac{1}{2}-\eps}w_n}_{\ell^2_n}\norm{w_n}_{\ell^\infty_n}^2\lesssim \norm{\langle n\rangle^{s-\frac{1}{2}-\eps}w_n}_{\ell^2_n}\norm{w_n}_{\ell^2_n}^2.
\end{equation*}
Putting everything together, we have shown
\begin{equation*}
|\mathcal{R}_1(w)| \lesssim \norm{w}_{H^{s-\frac{1}{2}-\eps}}^4 \|w\|_{L^{2}}^{2}.
\end{equation*}
\medskip 

\noi
$\bullet$ \textbf{Case 2:} $n_{\max}\in \{ |n_2|,|n_3|,|n|\}$

It suffices to assume $n_{\max}=|n_2|$ as the remaining cases follow analogously as below.
As before, $\max( |n_1|, |n_3|, |n|) \gtrsim |n_2|$. We apply the argument of Case 1 if $|n_1|\gtrsim |n_2|$.  Instead, if $|n_3|\gtrsim |n_2|$, the remaining case being similar, Young's inequality, the embedding $\ell^2_{n}\subset \ell^3_{n}$ and H\"{o}lder's inequality yield
\begin{align*}
|\mathcal{R}_1(w)| 
&\lesssim \sum_{\G(\cj n)} w_{n_1}^3 \langle n_2 \rangle^{s-\frac{1}{2}-\eps}w_{n_2}\langle n_3\rangle^{s-\frac{1}{2}-\eps}w_{n_3}w_{n}\\
&\lesssim \lVert w_n \rVert_{\ell_n^3}^{3} \lVert \langle n \rangle^{s-\frac{1}{2}-\eps} w_n\rVert_{\ell^2_n}^2 \lVert w_{n}\rVert_{\ell^1_n}\\
&\lesssim \|w\|_{L^2}^{3}\norm{w}_{H^{s-\frac{1}{2}-\eps}}^3 .
\end{align*}
This completes the estimate for $\mathcal{R}_1(w)$ and the estimate for $\mathcal{R}_2(w)$ is similar.
Thus, we have established \eqref{energyest} in the region (i).

\medskip 

\noi

We now move onto establishing \eqref{energyest} when $\tfrac 12<s\leq 1$. This regime is responsible for the regions (ii) and (iii) in \eqref{highalpha}.
As before, we begin with $\N_1(w)$. Notice that since $s\leq 1$, we can no longer apply Lemma~\ref{Lemma: DMVT application}.
We set $\s= s-\frac 12 -\eps$ and define $\tilde{w}_n = \jb{n}^{\s}w_n$.
Without loss of generality, we suppose $|n_1|\les |m_1|$.
The regularity restriction of $s>\frac 23$ arises from the following estimate:
\begin{align}
\bigg\| \sum_{\G(n_1)} \jb{m_1}^{\s-\frac 16} w_{m_1}\cj{w_{m_2}}w_{m_3} \bigg\|_{\l^{\infty}_{n_1}} \les \|w\|_{H^{\frac 16}}^{2}\|w\|_{H^{\s}} \les \|w\|_{H^{\s}}^{3}, \label{linfty3}
\end{align}
where the second inequality holds provided $s>\frac 23$.
We decompose the sum in $\N_1(w)$ into a few cases depending on which frequency attains $n_{\max}$.

\medskip

\noi 
$\bullet$ \textbf{Case 1:}  $|n|\sim |n_1| \gg |n_2|,|n_3|$

In this case, it is clear from \eqref{p1} and Lemma~\ref{lemma: phase lower bound} that 
\begin{align}
|\phi(\cj n)| \gtrsim n_{\max}^{2\al-1}|n-n_1|  \label{case1a}
\end{align}
and from \eqref{derivs} and the mean value theorem,
\begin{align}
|\psi_{s}(\cj n)|\les n_{\max}^{2s-1}|n-n_1|. \label{case1b}
\end{align}
Hence with \eqref{linfty3}, we have
\begin{align*}
\text{RHS of} \,\,\eqref{N1} & \les \sum_{\G(\cj n)}  \frac{n_{\max}^{2s-1}}{n_{\max}^{2\al-1}}  \frac{\tilde{w}_{n_2}  \tilde{w}_{n_3} \tilde{w}_{n}   }{\jb{n}^{\s}\jb{n_1}^{\s-\tfrac 16}  \jb{n_2}^{\s}\jb{n_3}^{\s}} \bigg\| \sum_{\G(n_1)} \jb{m_1}^{\s-\frac 16} w_{m_1}\cj{w_{m_2}}w_{m_3} \bigg\|_{\l^{\infty}_{n_1}} \\
& \les \|w\|_{H^{\s}}^{3} \sum_{\G (\cj n)} \frac{\tilde{w}_{n_2}  \tilde{w}_{n_3} \tilde{w}_{n}}{n_{\max}^{\nu}\jb{n_2}^{\s}\jb{n_3}^{\s}},
\end{align*}
where 
\begin{align*}
\nu = 2\al-2s+2\s-\frac 16 =2\al-\frac 76-2\eps >0.
\end{align*}
By the Cauchy-Schwarz inequality, we bound this by 
\begin{align*}
\|w\|_{H^{\s}}^{6} \bigg( \sum_{\G(\cj n)} \frac{1}{n_{\max}^{2\nu}\jb{n_2}^{2\s}\jb{n_3}^{2\s}}   \bigg)^{\frac 12}
\end{align*}
and we are done, provided we show
\begin{align}
 \sum_{\G(\cj n)} \frac{1}{n_{\max}^{2\nu}\jb{n_2}^{2\s}\jb{n_3}^{2\s}}    \les 1. \label{sum1}
\end{align}
For $\dl>0$ sufficiently small, we have 
\begin{align*}
\sum_{\G(\cj n)} \frac{1}{n_{\max}^{2\nu}\jb{n_2}^{2\s}\jb{n_3}^{2\s}}  \sim\sum_{\substack{n_1, \, n_2, \, n_3 \\ |n_2|,|n_3|\ll |n_1|   }}\frac{1}{\jb{n_1}^{1+\dl}\jb{n_2}^{1+\dl}\jb{n_3}^{1+\dl}} \frac{\jb{n_2}^{1+\dl -2\s}\jb{n_3}^{1+\dl-2\s}}{ \jb{n_1}^{2\nu-1-\dl}}.
\end{align*}
Thus, provided
\begin{align}
2\nu>1 \qquad \text{and} \qquad 4\s +2\nu >3,\label{nucond}
\end{align}
we have 
\begin{align*}
\frac{\jb{n_2}^{1+\dl -2\s}\jb{n_3}^{1+\dl-2\s}}{ \jb{n_1}^{2\nu-1-\dl}} \les \frac{1}{ \jb{n_1}^{2\nu+4\s-3-3\dl}}\les 1,
\end{align*}
and hence \eqref{sum1} follows. 
The first condition in \eqref{nucond} requires $\al>\tfrac 56$, while the last condition requires $s>\tfrac{11}{6}-\al$. 

\medskip

\noi 
$\bullet$ \textbf{Case 2:}  $|n|\sim |n_2| \gg |n_1|,|n_3|$

We have 
\begin{align}
|\phi( \cj n)| \gtrsim n_{\max}^{2\al}, \label{case2a}
\end{align}
and using \eqref{linfty3} leads to 
\begin{align*}
\text{RHS of} \,\, \eqref{N1} & \les \|w\|_{H^{\s}}^{3} \sum_{\G (\cj n)} \frac{\tilde{w}_{n_2}  \tilde{w}_{n_3} \tilde{w}_{n}}{n_{\max}^{\nu}\jb{n_1}^{\s -\tfrac 16}\jb{n_3}^{\s}} ,
\end{align*}
where 
\begin{align*}
\nu = 2\al-2s+2\s=2\al-1-2\eps>0.
\end{align*}
Using Cauchy-Schwarz as in the previous case, we sum over $n_1, n_3$ and $n_2$ provided
\begin{align*}
2\nu>1 \qquad \text{and} \qquad 2\nu+4\s-\frac 13 >3
\end{align*}
and hence $s>\tfrac{11}{6}-\al$. Notice the first condition above requires $\al>\tfrac 34$.

\medskip

\noi 
$\bullet$ \textbf{Case 3:}  $|n|\sim |n_3| \gg |n_1|,|n_2|$

We have 
\begin{align*}
|\phi(\cj n)|\gtrsim n_{\max}^{2\al-1}|n-n_3| \qquad \text{and} \qquad |\psi_{s}(\cj n)|\les n_{\max}^{2s-1}|n-n_3|.
\end{align*}
Thus 
\begin{align*}
\text{RHS of} \,\, \eqref{N1} & \les \|w\|_{H^{\s}}^{3} \sum_{\G (\cj n)} \frac{\tilde{w}_{n_2}  \tilde{w}_{n_3} \tilde{w}_{n}}{n_{\max}^{\nu}\jb{n_1}^{\s -\tfrac 16}\jb{n_2}^{\s}},
\end{align*}
where 
\begin{align*}
\nu = 2\al-2s+2\s=2\al -1-2\eps >0.
\end{align*}
Using Cauchy-Schwarz as in the previous cases, we sum over $n_1, n_2$ and $n_3$ provided
\begin{align*}
2\nu>1 \qquad \text{and} \qquad 2\nu+4\s-\frac 13 >3,
\end{align*}
and hence $s>\tfrac{11}{6}-\al$. The first condition above requires $\al>\tfrac 34$.

\medskip

\noi 
$\bullet$ \textbf{Case 4:}  $|n_1|\sim |n_2| \gg |n_3|,|n|$

We have 
\begin{align*}
|\phi(\cj n)|\gtrsim n_{\max}^{2\al-1}|n-n_3| \qquad \text{and} \qquad |\psi_s (\cj n)|\les n_{\max}^{2s-1}|n-n_3|
\end{align*}
and we proceed as in Case 1 as long as $s>\tfrac{11}{6}-\al$ and $\al>\tfrac 56$.

\medskip

\noi 
$\bullet$ \textbf{Case 5:}  $|n_1|\sim |n_3| \gg |n_2|,|n|$

We have 
\begin{align*}
|\phi(\cj n)|\gtrsim n_{\max}^{2\al},
\end{align*}
and hence
\begin{align*}
\text{RHS of} \,\, \eqref{N1} & \les \|w\|_{H^{\s}}^{3} \sum_{\G (\cj n)} \frac{\tilde{w}_{n_2}  \tilde{w}_{n_3} \tilde{w}_{n}}{n_{\max}^{\nu}\jb{n}^{\s}\jb{n_2}^{\s}}, 
\end{align*}
where 
\begin{align*}
\nu = 2\al-2s+2\s-\frac{1}{6}=2\al -\frac{7}{6}-2\eps >0.
\end{align*}
By Cauchy-Schwarz and summing in $n, n_2$ and $n_1$ as long as 
\begin{align*}
2\nu>1 \qquad \text{and}\qquad 2\nu+4\s>3,
\end{align*}
and hence $s>\tfrac{11}{6}-\al$. The first condition above is satisfied provided $\al>\tfrac 56$.

\medskip

\noi 
$\bullet$ \textbf{Case 6:}  $|n_2|\sim |n_3| \gg |n_1|,|n|$

We have 
\begin{align*}
|\phi(\cj n)|\gtrsim n_{\max}^{2\al-1}|n-n_1| \qquad \text{and} \qquad |\psi_{s}(\cj n)|\les n_{\max}^{2s-1}|n-n_1|,
\end{align*}
and we proceed as in Case 4 as long as $s>\tfrac{11}{6}-\al$ and $\al>\tfrac 34$.

\medskip 

\noi
$\bullet$ \textbf{Case 7:} $|n_1|\sim |n_2|\sim|n_3|\gg |n|$

From Lemma~\ref{lemma: phase lower bound}, we have $|\phi(\cj n)| \gtrsim n_{\max}^{2\al}$ and hence
\begin{align*}
\text{RHS of} \,\, \eqref{N1} & \les \|w\|_{H^{\s}}^{3} \sum_{\G (\cj n)} \frac{n_{\max}^{2s} \tilde{w}_{n_2}  \tilde{w}_{n_3} \tilde{w}_{n}}{n_{\max}^{2\al} n_{\max}^{3\s-\frac{1}{6}}\jb{n}^{\s}}.
\end{align*}
Applying Cauchy-Schwarz in $n_2,n_3$ and $n$, we sum provided 
\begin{align*}
2\bigg(2\al+3\s-\frac 16-2s\bigg)> 2 \qquad \text{and} \qquad 2\al-2s+3\s-\frac{1}{6}+\s>\frac{3}{2},
\end{align*}
which requires $\al +s > \tfrac{11}{6}$ and $2\al+s>\tfrac{8}{3}$. Notice that when $\al>\tfrac 56$, this latter condition is superseded by the former. 
The remaining cases of the form $|n_{j_1}|\sim |n_{j_2}|\sim |n_{j_3}|\gg |n_{j_4}|$ with distinct $j_{k}\in \{ 1,2,3,4\}$ ($n_{j_4}:=n$) are similar and are thus omitted.

\medskip

\noi 
$\bullet$ \textbf{Case 8:}  $|n|\sim |n_1|\sim |n_2|\sim |n_3|$

\noi
We distinguish when $\al$ is `close too' or `far from' $1$.

\medskip

\noi 
$\bullet$ \textbf{Subcase 8.1:} $1<\al <\tfrac{5}{4}$

\noi
By Lemma~\ref{lemma: phase lower bound}, we have
\begin{align*}
\sum_{\G(\cj n)} &\frac{n_{\max}^{2s}}{|n-n_1||n-n_3|n_{\max}^{2\al-2}\jb{n}^{\s}} \bigg(\prod_{j=1}^{3}\frac{1}{\jb{n_j}^{\s}} \bigg)\bigg(\sum_{\G(n_1)} \tilde{w}_{m_1} w_{m_2} w_{m_3} \bigg) \tilde{w}_{n_2} \tilde{w}_{n_3} \tilde{w}_{n} \\
& \les \sum_{\G(\cj n)} \frac{1}{|n-n_1||n-n_3|n_{\max}^{\nu}}  \bigg(\sum_{\G(n_1)} \tilde{w}_{m_1} w_{m_2} w_{m_3} \bigg) \tilde{w}_{n_2} \tilde{w}_{n_3} \tilde{w}_{n}, 
\end{align*}
where 
\begin{align}
\nu = 2\al-2-2s+4\s=2(\al+s-2-2\eps)>0, \label{nu7.2}
\end{align}
provided 
\begin{align}
\al+s>2. \label{nupos7.2}
\end{align}
By Cauchy-Schwarz in $n,n_1$ and $n_3$, followed by summing in $n$ and $n_3$, we get
\begin{align*}
&\les \|w\|_{H^{\s}}^{3} \bigg\|   \frac{1}{\jb{n-n_1}\jb{n-n_3} \jb{n_{\max}}^{\nu}}  \bigg(\sum_{\G(n_1)} \tilde{w}_{m_1} w_{m_2} w_{m_3} \bigg)    \bigg\|_{\l^{2}_{n,n_1,n_3}} \\
& \les  \|w\|_{H^{\s}}^{3} \bigg\|  \frac{1}{\jb{n_1}^{\nu}}  \sum_{\G(n_1)} \tilde{w}_{m_1} w_{m_2}w_{m_3}    \bigg\|_{\l^{2}_{n_1}}.
\end{align*}
Imposing
\begin{align}
\al+s <2+\frac{1}{4} \label{hold7.2}
\end{align}
 implies $\nu<\frac{1}{2}$ so that we can apply H\"{o}lder's inequality and then Young's inequality, with exponents\footnote{Here, we use the notation $a-$ (respectively, $a+$) to denote $a-\dl$ (respectively, $a+\dl$), where $0<\dl\ll 1$ is extremely small.}
 \begin{align*}
\frac{1-2\nu+}{2}+2=\frac{1}{2}+\frac{2}{\frac{2}{2-\nu+} },
\end{align*}
     to obtain
\begin{align*}
&\les \|w\|_{H^{\s}}^{3} \bigg\| \sum_{\G(n_1)} \tilde{w}_{m_1} w_{m_2} w_{m_3}  \bigg\|_{\l_{n_1}^{ \frac{2}{1-2\nu+}   }} \\
& \les \|w\|_{H^{\s}}^{4} \|w_n\|_{\l_{n}^{  \frac{2}{2-\nu+} }}^{2}.
\end{align*}
Once more with H\"{o}lder's inequality, we have
\begin{align*}
\les \|w\|_{H^{\s}}^{6}\| \jb{n}^{-\s}\|^{2}_{\l_{n}^{ \frac{2}{1-\nu+}}} \les \|w\|_{H^{\s}}^{6},
\end{align*}
provided 
\begin{align*}
\frac{2\s}{1-\nu+}>1.
\end{align*}
Using \eqref{nu7.2}, this last condition requires 
\begin{align}
s>\frac{3-\al}{2}. \label{c37.2}
\end{align}

Putting the conditions \eqref{nupos7.2}, \eqref{hold7.2} and \eqref{c37.2} together implies we must enforce in this subcase 
\begin{align*}
\max\bigg( \frac{1}{2}, \frac{3-\al}{2}, 2-\al \bigg) <s \leq \min\bigg(1, \frac{9}{4}-\al \bigg),
\end{align*}
 where the upper bound is strict if the minimum is $\tfrac 94 -\al$. 
Now as $\al\leq \tfrac{5}{4}$, $\min\big(1, \tfrac{9}{4}-\al \big)=1$ and this implies the range
\begin{align*}
\frac{3-\al}{2} <s \leq 1. 
\end{align*}

\medskip

\noi 
$\bullet$ \textbf{Subcase 8.2:} $\al \geq \frac{5}{4}$

Given $n\in \Z$, let
\begin{align*}
\G(n,\rho)=\G(n)\cap \{ (n_1,n_2,n_3)\in \Z^{3}\,:\, \rho=(n-n_1)(n-n_3)\in \Z \}.
\end{align*}
From \eqref{divisorcount} we have, for any $\dl>0$, there exists a $C_\dl >0$ such that 
\begin{align}
|\#\G(n,\rho)| \les C_{\dl}|\rho|^{\dl}. \label{divcount}
\end{align}
By Lemma~\ref{lemma: phase lower bound} and \eqref{linfty3}, we have
\begin{align*}
|\N_1(w)| & \les \sum_{n}\sum_{\rho\neq 0} \sum_{\G(n,\rho)} \frac{n_{\max}^{2s}}{|\rho| n_{\max}^{2\al-2}}  \frac{\tilde{w}_{n_2}  \tilde{w}_{n_3} \tilde{w}_{n}}{\jb{n}^{\s}\jb{n_1}^{\s-\tfrac 16}  \jb{n_2}^{\s}\jb{n_3}^{\s}} \bigg\| \sum_{\G(n)} \jb{m_1}^{\s-\frac 16} w_{m_1}\cj{w_{m_2}}w_{m_3} \bigg\|_{\l^{\infty}_{n_1}} \\
& \les \|w\|_{H^\s}^{3} \sum_{n}\sum_{\rho\neq 0} \sum_{\G(n,\rho)} \frac{1}{|\rho|n_{\max}^{\nu}} \tilde{w}_{n_2}  \tilde{w}_{n_3} \tilde{w}_{n},
\end{align*}
where 
\begin{align*}
\nu= 2\al-2-2s+4\s-\frac{1}{6}=2\al+2s-\frac{25}{6}-2\eps,
\end{align*}
which is positive provided 
\begin{align}
\al+s>\frac{25}{12}. \label{simcond}
\end{align}
To continue, we follow the argument in \cite[Proposition 6.1]{oh2016quasi} and for completeness we detail it here. By Cauchy-Schwarz, \eqref{divcount} and Lemma~\ref{lemma:sumestimate}, we have
\begin{align*}
& \les \|w\|_{H^{\s}}^{4} \bigg[ \sum_{n} \bigg(  \sum_{\rho\neq 0} \sum_{\G(n,\rho)} \frac{1}{|\rho|n_{\max}^{\nu}} \tilde{w}_{n_2}  \tilde{w}_{n_3}  \bigg)^{2}  \,   \bigg]^{\frac{1}{2}} \\ 
& \les \|w\|_{H^{\s}}^{4} \bigg[ \sum_{n} \bigg( \sum_{\rho\neq 0} \frac{1}{|\rho|^{1+2 \dl} } \sum_{\G(n,\rho)} 1 \bigg) \sum_{\rho \neq 0} \sum_{\G(n,\rho)} \frac{1}{|\rho|^{1-2\dl} n_{\max}^{2\nu}} \tilde{w}_{n_2} ^{2} \tilde{w}_{n_3}^{2}    \bigg]^{\frac{1}{2}} \\
& \les \|w\|_{H^{\s}}^{4} \bigg(  \sum_{ n_2, n_3} \tilde{w}_{n_2}^{2}\tilde{w}_{n_3}^{2} \sum_{n_{1}\neq n_2} \frac{1}{|n_1-n_2|^{1-2\dl}\jb{n_1}^{2\nu}} \bigg)^{\frac{1}{2}} \\
& \les \|w\|_{H^{\s}}^{6},
\end{align*}
where from \eqref{simcond} we choose $\dl>0$ small enough so that $\dl<\nu$.

The required range for $s$ in this subcase is 
\begin{align*}
\max\bigg(\frac{2}{3}, \frac{25}{12}-\al \bigg) <s \leq 1. 
\end{align*}
This completes the estimates for $\N_{1}(w)$.
The estimate for $\N_2(w)$ follows analogously.

We now move onto bounding $\mathcal{R}_{1}(w)$. Writing 
\begin{align*}
m(\cj{n}):= \frac{|\psi_{s}(\cj n)| }{|\phi(\cj{n})| \jb{n_1}^{3\s}\jb{n_2}^{\s} \jb{n_3}^{\s}\jb{n}^{\s}},
\end{align*}
it suffices to show
\begin{align}
 \sum_{\G(\cj n)} m(\cj{n})|\tilde{w}_{n_1}|^{3} \tilde{w}_{n_2} \tilde{w}_{n_3}\tilde{w}_{n} \les \|\tilde{w}\|_{L^{2}}^{6}. \label{r1bit}
\end{align}
As above, we divide into a few cases. 

\medskip

\noi 
$\bullet$ \textbf{Case 1:} $|n|\sim |n_1|\gg |n_2|,|n_3|$

Using \eqref{case1a} and \eqref{case1b} we have 
\begin{align*}
m(\cj n) \les \frac{1}{n_{\max}^{\nu}\jb{n_2}^{\s}\jb{n_3}^{\s}},
\end{align*}
where $\nu=2\al-2s+4\s=2\al+2s-2-4\eps$. By Cauchy-Schwarz in $n_2$, $n_3$ and $n$ and the embedding $\l^{2}_{n}\subset \l^6_{n}$, we have
\begin{align*}
\text{LHS of } \,\, \eqref{r1bit} \les \|\tilde{w}\|^{5}_{L^2} \bigg( \sum_{n} \tilde{w}_{n}^{2} \sum_{n_2,n_3} \frac{1}{ n_{\max}^{2\nu}\jb{n_2}^{2\s}\jb{n_3}^{2\s}} \bigg)^{\frac 12} \les \|\tilde{w}\|_{L^2}^{6},
\end{align*}
where we can sum provided $4\s+2\nu>2$ which requires $s>1-\tfrac{1}{2}\al$.

\medskip

\noi 
$\bullet$ \textbf{Case 2:} $|n|\sim |n_3|\gg |n_1|,|n_2|$

Using \eqref{case2a}, we have 
\begin{align*}
m(\cj n) \les \frac{1}{n_{\max}^{\nu}\jb{n_1}^{3\s}\jb{n_2}^{\s}},
\end{align*}
where $\nu=2\al-2s+2\s=2\al-1-2\eps$. By Cauchy-Schwarz in $n_1$, $n_2$ and $n_3$ and the embedding $\l^{2}_{n}\subset \l^6_{n}$, we have
\begin{align*}
\text{LHS of } \,\, \eqref{r1bit} \les \|\tilde{w}\|_{L^2}^5 \bigg( \sum_{n_3} \tilde{w}_{n_3}^{2} \sum_{n_1,n_2} \frac{1}{ n_{\max}^{2\nu}\jb{n_1}^{6\s}\jb{n_2}^{2\s}} \bigg)^{\frac 12} \les \|\tilde{w}\|_{L^2}^{6},
\end{align*}
where we can sum provided $s>\max(\tfrac 23,2-2\al)$.

\medskip

\noi 
$\bullet$ \textbf{Case 3:} $|n_1|\sim |n_2| \sim |n_3|\sim |n|$

From \eqref{p1}, we have 
\begin{align*}
m(\cj n) \les \frac{n_{\max}^{2s}}{|n-n_1||n-n_3| n_{\max}^{2\al-2}}\frac{1}{ \jb{n_1}^{3\s}\jb{n_2}^{\s}\jb{n_3}^{\s}\jb{n}^{\s}} \sim \frac{1}{|n-n_1||n-n_3| n_{\max}^{\nu}},
\end{align*}
where 
\begin{align*}
\nu = 2\al-2-2s+6\s>0,
\end{align*}
provided $s>\tfrac 54-\tfrac 12 \al$. An application of Cauchy-Schwarz then implies 
\begin{align*}
\text{LHS of } \,\, \eqref{r1bit} \les \|\tilde{w}\|_{L^2}^5 \bigg(\sum_{n} \tilde{w}_{n}^{2} \sum_{n_1,n_3\neq n} \frac{1}{|n-n_1|^{2}|n-n_3|^{2}}  \bigg)^{\frac 12} \les \| \tilde{w}\|_{L^{2}}^{6}.
\end{align*}

All remaining cases follow analogously from the methods in either Case 1 or Case 2 above and are thus omitted. This completes the bound for $\mathcal{R}_1$.
Notice the condition from Case 3 supersedes the conditions from Cases 1 and 2. Furthermore, at least for every $\al\geq \tfrac{7}{6}$, we have 
\begin{align*}
\frac{2}{3}\geq \frac{5}{4}-\frac{1}{2}\al
\end{align*}
and hence we obtain the final conditions on $s$ and $\al$ of \eqref{highalpha}.

Finally, this completes the proof of \eqref{energyw}.

\end{proof}

We also have the following difference estimate for $R_s(v)$.
It will be convenient to view $R_{s}$ as a multi-linear functional
\begin{align*}
R_{s}(u^{(1)},u^{(2)},u^{(3)},u^{(4)}):=\frac{1}{2}\text{Re} \, \sum_{\G(\cj n)}\frac{\psi_{s}(\cj n)}{\phi(\cj n)}u^{(1)}_{n_1}\cj{u^{(2)}_{n_2}}u^{(3)}_{n_3}\cj{u^{(4)}_{n}},
\end{align*}
where $R_s (v,v,v,v)=R_s (v)$.

\begin{proposition}\label{Prop: difference estimate for R_s(v)}
Suppose
 \begin{align}
 \begin{split}
\textup{(i)} \hphantom{xxXXXXXXXX} s>1, &\qquad \textup{when}\qquad \al>\frac{1}{2}, \,\,\textup{or} \\
\textup{(ii)} \qquad s>\max\bigg(2-\al,\frac 12 \bigg), &\qquad \textup{when}\qquad  \al\geq 1.
\end{split} \label{highalpha2}
\end{align}
Then, for sufficiently small $\eps>0$ there exists $C>0$ such that 
\begin{equation*}
|R_s(u)-R_s(v)|\leq C\norm{u-v}_{H^{s-\frac{1}{2}-\eps}}(\norm{u}_{H^{s-\frac{1}{2}-\eps}}^3+\norm{v}_{H^{s-\frac{1}{2}-\eps}}^3)
\end{equation*}
for all $u,v\in H^{s-\frac{1}{2}-\eps}(\T)$.
\end{proposition}

\begin{proof}
By the multi-linearity of $R_{s}(u)$, it suffices to show 
\begin{align}
|R_{s}(\{ u^{(j)}\}_{j=1}^{4})| \les \prod_{j=1}^{4} \| u^{(j)}\|_{H^{\s}},\label{Rsboundmult}
\end{align}
where $\s:= s-\tfrac 12 -\eps$.

We first consider case (i) in \eqref{highalpha2}. 
Using \eqref{RS} and Lemmas \ref{lemma: phase lower bound} and \ref{Lemma: DMVT application} we have
\begin{align*}
|R_s( \{ u^{(j)}\}_{j=1}^{4} )|\lesssim \sum\limits_{\G(\cj n)}\langle n_{\max}\rangle^{2s-2\alpha}|u^{(1)}_{n_1}||u^{(2)}_{n_2}| |u^{(3)}_{n_3}| |u^{{4})}_{n}|.
\end{align*}
Similar to the proof of \eqref{Prop: Energy Estimate}, it suffices to consider the following case when $ n_{\max}=|n_1|$ and $|n_2|\gtrsim |n_1|$.
Using Young's and H\"{o}lder's inequality, we get
\begin{align*}
|R_s(\{ u^{(j)}\}_{j=1}^{4})|&\les \norm{\langle n_1 \rangle^{\s}u^{(1)}_{n_1}}_{\ell^2_{n_1}}\norm{\langle n_2\rangle^{\s} u^{(2)}_{n_2} }_{\ell^2_{n_2}}\| u^{(3)}_{n_3}\|_{\l^{1}_{n_3}} \| u^{(4)}_{n}\|_{\l^{1}_{n}}  
\lesssim \prod_{j=1}^{4} \| u^{(j)}\|_{H^{\s}}.
\end{align*}
We now consider case (ii) in \eqref{highalpha2}.
As it is already contained within the case $s>1$ and $\al >\frac 12$ proved above, we now bound 
$|R_{s}(\{ u^{(j)}\}_{j=1}^{4} )|$ when $s\leq 1$ for $\al \geq 1$. Given such an $s$, let $\s = s-\tfrac 12 -\eps$. 
We have
\begin{align}
|R_s (\{ u^{(j)}\}_{j=1}^{4})| \les \sum_{\G(\cj n)} m(\cj n) |\tilde{u}^{(1)}_{n_1}| |\tilde{u}^{(2)}_{n_2}||\tilde{u}^{(3)}_{n_3}|\tilde{u}^{(4)}_{n}|,  \label{r1bd}
\end{align}
where 
\begin{align*}
m(\cj n)=\frac{|\psi_{s}(\cj n)|}{|\phi(\cj n)|\jb{n}^{\s}} \prod_{j=1}^{3}\frac{1}{\jb{n_j}^\s}.
\end{align*}
As in the proof of Proposition~\ref{Prop: Energy Estimate}, we consider a few cases depending on $n_{\max}$. 

\medskip

\noi 
$ \bullet$ \textbf{Case 1:} $|n|\sim|n_1|\gg |n_2|,|n_3|$

In this case, we have
\begin{align*}
|\phi(\cj n)|\gtrsim n_{\max}^{2\al-1}|n-n_1| \qquad \text{and} \qquad |\psi(\cj n)|\les n_{\max}^{2s-1}|n-n_1|,
\end{align*}
 and thus
\begin{align*}
m(\cj n) \les \frac{1}{n_{\max}^{\nu}} \frac{1}{\jb{n_2}^{\s}\jb{n_3}^{\s}},
\end{align*}
where $\nu =2\al-1-2\eps>0$.
By Cauchy-Schwarz inequality, 
\begin{align*}
\text{RHS of} \,\, \eqref{r1bd} &\les \prod_{j=1}^{3}\|u^{(j)}\|_{H^{\s}}^{3} \bigg( \sum_{n,n_2,n_3 \in \Z}\frac{1}{n_{\max}^{2\nu}\jb{n_2}^{2\s}\jb{n_3}^{2\s}}(\tilde{u}^{(4)}_{n})^{2}  \bigg)^{\frac 12}
\les \prod_{j=1}^{4}\|u^{(j)}\|_{H^{\s}}^{4},
\end{align*}
where we can sum in $n_2$ and $n_3$ provided $\nu +2\s>1$, which requires
\begin{align*}
\al+s>\frac{3}{2}. 
\end{align*}

\medskip

\noi 
$ \bullet$ \textbf{Case 2:} $|n|\sim|n_2|\gg |n_1|,|n_3|$

Here we use $|\phi(\cj n)|\gtrsim n_{\max}^{2\al}$ which implies 
\begin{align*}
m(\cj n) \les \frac{1}{n_{\max}^{\nu}\jb{n_1}^{\s}\jb{n_3}^{\s}},
\end{align*}
where $\nu = 2\al-1-2\eps>0$. We proceed as in Case 1 by using Cauchy-Schwarz and summing in $n_1$ and $n_3$ with $\tilde{v}_{n_2}^{2}$ absorbing the remaining $n_2$ summation. 

It is easy to check that all remaining Cases 3 through 7 as explicated in the proof of Proposition~\ref{Prop: Energy Estimate} follow analogously to the two cases above. 

\medskip

\noi 
$ \bullet$ \textbf{Case 3:} $|n|\sim |n_1| \sim |n_2| \sim |n_3|$

We can only use the lower bound of Lemma~\ref{lemma: phase lower bound} and this implies 
\begin{align*}
m(\cj n) \les \frac{1}{  |n-n_1| |n-n_3| n_{\max}^{\nu}},
\end{align*}
where $\nu=2\al-2-2s+4\s$ which is non-negative provided
\begin{align*}
\al+s> 2. 
\end{align*} 
Then Cauchy-Schwarz over $\G(\cj n)$ gives
\begin{align*}
\text{RHS of} \,\, \eqref{r1bd} &\les \prod_{j=1}^{3}\|u^{(j)}\|_{H^{\s}}^{3} \bigg( \sum_{n\in \Z}  (\tilde{u}^{(4)}_{n})^{2}  \sum_{n_1, n_3 \in \Z}  \frac{1}{|n-n_1|^{ 2}|n-n_3|^2}\bigg)^{\frac 12} \\
& \les \prod_{j=1}^{4}\|u^{(j)}\|_{H^{\s}}^{4}.
\end{align*}
This completes the proof of \eqref{Rsboundmult}.

\end{proof}

\begin{remark} \rm 
Since $2-\al \leq \frac{25}{12}-\al$ and $2-\al \leq \tfrac{3-\al}{2}$ for $\al \geq 1$, the restriction \eqref{highalpha} for the energy estimate supersedes \eqref{highalpha2}~(ii), which is the condition for the correction term $R_{s}(u)$.
\end{remark}

\subsection{Proof of Theorem~\ref{Main result} \eqref{resultals1}} \label{subsec:concludeQI1}
In this subsection, we follow the argument introduced in \cite{PVT} to conclude quasi-invariance of Gaussian measures $\mu_{s}$ for $\al >\tfrac{1}{2}$ and those $s$ given in Proposition~\ref{Prop: Energy Estimate}. In particular, we conclude Theorem~\ref{Main result}~\eqref{resultals1} and the $s>1$ portion of Theorem~\ref{thm:localqi}. 

We define the following measures:
\begin{align*}
d\rho_{s}=F_{s}(u)d\mu_s \qquad \text{and} \qquad d\rho_{s,N}=F_{s,N}(u)d\mu_s,\quad 
\end{align*}
where 
\begin{align*}
F_s(u)=\exp\left(-\tfrac{1}{2}E_s(u)+\tfrac{1}{2}\norm{u}_{H^s}^2  \right)= \exp\left(-\tfrac{1}{2}R_s(u)\right)\quad\textnormal{ and }\quad F_{s,N}(u)=F_s(\mathbf{P}_{\leq N} u).
\end{align*}
The measure $\rho_{s,N}$ can also be expressed as
\begin{equation}\label{alternative expression}
d\rho_{s,N}=Z_{s,N}^{-1}\exp\left(-\tfrac{1}{2}E_s(\mathbf{P}_{\leq N} u) \right)du_{\leq N}\times d\mu_{s,N}^\perp,
\end{equation}
where $du_{\leq N}$ denotes the Lebesgue measure on $\mathbb{C}^{2N+1}$.
The constant $Z_{s,N}^{-1}$ is the normalisation constant associated to the measure $\mu_{s,N}$ which is given by
\begin{align*}
d\mu_{s,N}=Z^{-1}_{s,N}e^{-\frac{1}{2}\norm{\mathbf{P}_{\leq N}u}_{H^s}^2}du_{N}.
\end{align*}
In particular, $\mu_{s,N}$ is the probability measure induced under the map
\begin{align*}
\omega\in\Omega \longmapsto u_{\leq N}^\omega(x)=\sum\limits_{|n| \leq N}\frac{g_n(\omega)}{\langle n\rangle^s}e^{inx}.
\end{align*}
Likewise, $\mu^{\perp}_{s,N}$ is the probability measure induced under the map
\begin{align*}
\omega\in\Omega \longmapsto u_{>N}^\omega(x)=\sum\limits_{|n|> N}\frac{g_n(\omega)}{\langle n\rangle^s}e^{inx}.
\end{align*}
Note that we do not require $F_{s}$ and $F_{s,N}$ to be integrable with respect to $\mu_s$. Hence $\rho_s$ and $\rho_{s,N}$ are not necessarily probability measures. However, as the quasi-invariance argument is purely local (see the proof of Theorem 1.2 (i) below), it suffices to have $F_{s,N}\in L^{1}_{\text{loc}}(\mu_{s})$ and with convergence to $F_{s}$. This is the content of the next proposition, whose proof can be found in \cite[Proposition 2.1]{PVT}.

\begin{proposition}\label{convergence on bounded sets}
Let $s$ be as in \eqref{highalpha2}. Then, for every bounded set $A\subset H^{s-\frac{1}{2}-\eps}(\T)$, we have
\begin{align*}
\lim\limits_{N\rightarrow\infty}\int_{A}|F_{s,N}(u)-F_s(u)|d\mu_s(u)=0
\end{align*}
and in particular
\begin{align*}
\lim\limits_{N\rightarrow\infty}|\rho_{s,N}(A)-\rho_s(A)|=0
\end{align*}
\end{proposition}

The next result states important properties of the truncated flow $\Phi_{N}$.

\begin{proposition}\label{Prop: GWP growth bound}
Let $s$ be as in Proposition~\ref{prop: gwp} be such that the flow $\Phi$ of FNLS~\eqref{FNLS} is globally well-defined.  Then, the following statements hold: 
\begin{enumerate}[\normalfont (i)]
\setlength\itemsep{0.3em}
\item  For every $R>0$ and $T>0$, there exists $C(R,T) >0$ such that 
\begin{align*}
\Phi_N(t)\left(B_R \right)\subset B_{C(R,T)}
\end{align*}
for all $t\in[0,T]$ and for all $N\in \mathbb{N}\cup\{\infty\}$. Here, $\Phi_{\infty}:=\Phi$ denotes the untruncated flow. 
\item Let $A\subset H^{s-\frac{1}{2}-\eps}(\T)$ be a compact set and $t\in\R$. Then, for every $\delta>0$, there exists $N_0\in\mathbb{N}$ such that 
\begin{align*}
\|\Phi(t)(u)-\Phi_{N}(t)(u)\|_{H^{s}}<\dl,
\end{align*}
for any $u\in A$ and any $N\geq N_{0}$. Furthermore, we have
\begin{align*}
\Phi(t)\left(A  \right)\subset \Phi_N (t)\left(A+B_\delta\right)
\end{align*}
 for all $N\geq N_0$.
\end{enumerate}
\end{proposition}

We also have the following local-in-time version of Proposition~\ref{Prop: GWP growth bound}.

\begin{proposition}\label{Prop: LWP growth bound}
Let $s$ be as in Proposition~\ref{prop: gwp} \textup{(ii)} be such that the flow $\Phi$ of FNLS~\eqref{FNLS} is only locally well-defined. Then, the following statements hold: 
\begin{enumerate}[\normalfont (i)]
\setlength\itemsep{0.3em}
\item  Then, for every $R>0$, there exist $T(R)>0$ and $C(R) >0$ such that 
\begin{align*}
\Phi_N(t)\left(B_R \right)\subset B_{C(R)}
\end{align*}
for all $t\in[0,T(R)]$ and for all $N\in \mathbb{N}\cup\{\infty\}$. 
\item Let $A\subset B_{R} \subset H^{s-\frac{1}{2}-\eps}(\T)$ be a compact set and denote by $T(R)>0$ the local existence time of the solution map $\Phi$ defined on $B_{R}$. Then, for every $\delta>0$, there exists $N_0\in\mathbb{N}$, such that 
\begin{align*}
\|\Phi(t)(u)-\Phi_{N}(t)(u)\|_{H^{s}}<\dl,
\end{align*} for any $u\in A$, $N\geq N_0$ and $t\in [0,T(R)]$. Furthermore, we have
\begin{align*}
\Phi(t)\left(A  \right)\subset \Phi_N (t)\left(A+B_\delta\right)
\end{align*}
for all $t\in[0,T(R)]$ and for all $N\geq N_0$.
\end{enumerate}

\end{proposition}

The proof of Proposition~\ref{Prop: GWP growth bound}~(i) follows from the global well-posedness of FNLS~\eqref{FNLS} when $\al>\tfrac{10\al+1}{12}$, while the proof of Proposition~\ref{Prop: LWP growth bound}~(i) follows by the local existence theory (for short times) when $\tfrac 12 < \al \leq \tfrac{10\al+1}{12}$. The proof of Proposition~\ref{Prop: GWP growth bound}~(ii) follows from the arguments in \cite[Appendix B]{oh2016quasi} using the existence theory in Appendix~\ref{app:gwp}, \cite{Bou} and \cite{Cho}.

\begin{proof}[Proof of Theorem~\ref{Main result} (i)]
In the following we fix $s$ and $\alpha$ satisfying the conditions of Propositions \ref{Prop: Energy Estimate} and \ref{Prop: difference estimate for R_s(v)}. As long as the conclusions of these propositions are satisfied, the following general argument due to \cite{PVT} implies the quasi-invariance of $\mu_{s}$ (either globally or locally in time). For clarity, we will only detail the following arguments in the case when FNLS~\eqref{FNLS} admits a globally well-defined flow $\Phi$ (see Proposition~\ref{prop: gwp}). We obtain local-in-time quasi-invariance from the same arguments by suitably restricting to the local well-posedness lifetime where necessary. 

Given $t>0$, by the inner regularity of the measure $\mu_s$, it is enough to show that
\begin{equation}\label{quasi for ms}
A\subset H^{s-\frac{1}{2}-\eps}\textnormal{ compact and }\mu_s(A)=0 \implies \mu_s(\Phi(-t)A)=0.
\end{equation}
From Proposition \ref{Prop: difference estimate for R_s(v)} with $u=0$, we have $0<\exp\left(R_s(v)\right)<\infty$ for almost all $v\in A$. Hence the implication \eqref{quasi for ms} is equivalent to the following implication:
\begin{align*}
A\subset H^{s-\frac{1}{2}-\eps}\textnormal{ compact and }\rho_s(A)=0 \implies \rho_s(\Phi(-t)A)=0.
\end{align*}
As $A$ is compact, there exists $R>0$ such that $A\subset B_R$. Then, by Proposition~\ref{Prop: GWP growth bound}, there exists a constant $C(R)>0$ such that
\begin{equation}\label{conseq of LWP bound}
\Phi(\tau)\left(B_{2R}\right)\cup \Phi_{N}(\tau)\left(B_{2R}\right)\subset B_{C(R)}
\end{equation}
for all $\tau \in [0,t]$.
For a measurable $D\subset B_{2R}$, it follows from \eqref{alternative expression}, Liouville's theorem and the invariance of complex Gaussians under rotations, that
\begin{align*}
\left| \frac{d}{d\tau}\rho_{s,N}(\Phi_N(\tau)(D))  \right| &= \left| \frac{d}{d\tau} Z_{s,N}^{-1}\int_{\Phi_N(\tau)(D)}\exp\left(-\tfrac{1}{2}E_s(\mathbf{P}_{\leq N} u) \right)du_{\leq N}\times d\mu_{s,N}^\perp  \right| \\
&=\left| Z_{s,N}^{-1}\int_{D}\frac{d}{d\tau}\exp\left(-\tfrac{1}{2}E_s(\Phi(\tau)(\mathbf{P}_{\leq N} u)) \right)du_{\leq N}\times d\mu_{s,N}^\perp  \right|.
\end{align*}
Using the energy estimate of Proposition \ref{Prop: Energy Estimate} along with \eqref{conseq of LWP bound} we have
\begin{align*}
\left|\frac{d}{d\tau}\exp\left(-\tfrac{1}{2}E_s(\Phi(\tau)(\mathbf{P}_{\leq N} u)) \right)\right|\leq C(R) \exp\left(-\tfrac{1}{2}E_s(\Phi(\tau)(\mathbf{P}_{\leq N} u))\right)
\end{align*}
for all $\tau\in [0,t]$ and for all $u\in D$.
Combining the above we have
\begin{align*}
\left| \frac{d}{d\tau}\rho_{s,N}(\Phi_N(\tau)(D))  \right|&\leq Z^{-1}_{s,N}C(R)\int_{D}\frac{d}{d\tau}\exp\left(-\tfrac{1}{2}E_s(\Phi(\tau)(\mathbf{P}_{\leq N} u)) \right)du_{\leq N}\times d\mu_{s,N}^\perp\\
&=C(R)\rho_{s,N}(\Phi_N(\tau)(D)).
\end{align*}
From Gronwall's inequality, we get 
\begin{equation}\label{Gronwall consequence}
\rho_{s,N}(\Phi_N(\tau)(D))\leq e^{C(R)\tau}\rho_{s,N}(D)
\end{equation}
for all $N\in N$ and for all $\tau\in[0,t]$.
By Proposition \ref{Prop: GWP growth bound}~(ii), we have
\begin{align*}
\rho_{s}(\Phi_N(\tau)(A)) \leq \rho_{s}\left(\Phi_N(\tau)(A+B_\delta)\right)
\end{align*}
for any fixed $\delta>0$ and $N$ large enough. Further, from Proposition \ref{convergence on bounded sets} for $N$ large enough, we have
\begin{align*}
\rho_{s}\left(\Phi_N(\tau)(A+B_\delta)\right) \leq \rho_{s,N}\left(\Phi_N(\tau)(A+B_\delta)\right)+\delta
\end{align*}
and so
\begin{align*}
\rho_{s}(\Phi_N(\tau)(A))\leq \rho_{s,N} (\Phi_N(\tau)(A+B_\delta))+\delta.
\end{align*}
Choosing $\delta<R$ so that $A+B_\delta\subset B_{2R}$ and \eqref{Gronwall consequence}, can be applied we get
\begin{align*}
\rho_{s}(\Phi_N(\tau)(A))\leq e^{C(R)\tau}\rho_{s,N}(A+B_\delta)+\delta.
\end{align*}
Using Proposition \ref{convergence on bounded sets} to go from $\rho_{s,N}$ back to $\rho_s$, we have
\begin{equation}
\rho_{s}(\Phi_N(t)(A))\leq e^{C(R)t}\rho_{s}(A+B_\delta)+2\delta.
\end{equation}
Letting $\delta$ approach $0$ and using regularity properties of the measure $\mu_s$, we finally obtain
\begin{align*}
\rho_{s}(\Phi_N(\tau)(A))\leq e^{C(R)\tau}\lim\limits_{\delta\rightarrow 0} \rho_{s}(A+B_\delta)=e^{C(R)\tau}\rho_{s}(A)=0
\end{align*}
for any $\tau\in [0,t]$.
This completes the proof.
\end{proof}

\section{Improvement for $\al >\frac{5}{6}$}

In this section, we employ the hybrid argument (Method 4) from \cite{hybrid} in order to lower the regularity threshold we previously obtained using Method 3.
Namely, we complete the proofs of Theorem~\ref{Main result}~\eqref{resultals} and Theorem~\ref{thm:localqi} by proving the quasi-invariance of Gaussian measures $\mu_{s}$ under the flow of FNLS~\eqref{FNLS} for regularities satisfying \eqref{M4REG}.

\subsection{Alternative energy estimate}

Our first port of call is to obtain an energy estimate where we place two factors into the Fourier-Lebesgue space $\FL^{\s,\infty}(\T)$, where $\s<s$. By placing these two factors into this stronger norm, we can lower the regularity restriction; compare \eqref{highalpha} and \eqref{scond}.

\begin{proposition} \label{prop:weakenergyest}
Let $\al>\frac{5}{6}$ and 
\begin{align}
\max\bigg( \frac{2}{3}, \frac{11}{6}-\al \bigg) < s\leq 1. \label{scond}
\end{align}
Then, for sufficiently small $\eps>0$, there exists $C>0$ such that 
\begin{align}
\bigg|  \frac{d}{dt}E_{s}(\P_{\leq N}v(t))\Big\vert_{t=0}    \bigg|  \leq C\|\P_{\leq N}v(0)\|_{\FL^{s-\tilde{\eps},\infty}}^{2} \|\P_{\leq N}v(0)\|_{H^{s-\frac{1}{2}-\eps}}^{4}, \label{energyFLq2}
\end{align}
for any $N\in \mathbb{N}$, any solution $v$ to \eqref{Truncated equation for v} and for any $0<\tilde{\eps}<\eps$, uniformly in $t\in \R$.
\end{proposition}

\begin{proof}
Using \eqref{modifiedenergy}, the estimate \eqref{energyFLq2} reduces to proving that for small $\eps>0$, there exists $C>0$ such that
\begin{align}
\bigg|  \frac{d}{dt}E_{s,t}(\P_{\leq N}w(t))\Big\vert_{t=0}    \bigg|  \leq C\|\P_{\leq N}w(0)\|_{\FL^{s-\tilde{\eps},\infty}}^{2} \|\P_{\leq N}w(0)\|_{H^{s-\frac{1}{2}-\eps}}^{4}. \label{energyFLq1}
\end{align}
From \eqref{derivenergy}, \eqref{N1}, \eqref{R1}, \eqref{N2} and \eqref{R2}, \eqref{energyFLq1}  further reduces to showing
\begin{equation}
\bigg| \sum_{j=1}^{2} \mathcal{N}_{j}(\P_{\leq N}w(0))+\sum_{j=1}^{2}\mathcal{R}_{j}( \P_{\leq N}w(0))\bigg| \leq C\|\P_{\leq N}w(0)\|_{\FL^{s-\tilde{\eps},\infty}}^{2} \norm{\P_{\leq N}w(0)}_{H^{s-\frac{1}{2}-\eps}}^4
\label{energyFLq}
\end{equation}
for all $N\in \mathbb{N}$ and uniformly in $t\in \R$.
Recalling the decomposition \eqref{derivenergy}, we estimate $\mathcal{N}_1$ and $\mathcal{R}_1$, with estimates for $\mathcal{N}_{2}$ and $\mathcal{R}_2$ following analogously. In the following, we simply replace $w(0)$ by $w$.
We consider $\mathcal{N}_{1}$ first. 
Recall that in Cases 1 through 7 of the proof of Proposition~\ref{Prop: Energy  Estimate}~(ii) and (iii), we obtained 
 \begin{align}
|\N_{1}(w)|\les \|w\|_{H^{s-\frac 12 -\eps}}^{6} \label{61}
\end{align}
for any $\al > \frac{5}{6}$ and for any $s$ satisfying 
\begin{align}
1\geq s>\max \bigg( \frac 23 , \frac{11}{6}-\al \bigg). \label{rest1}
\end{align}
 Then in these cases, we obtain \eqref{energyFLq} by using the embedding \eqref{FLembed} to put two factors of \eqref{61} into the required Fourier-Lebesgue space. Note that we could certainly improve upon the regularity lower bound on $s$ in these cases by proving \eqref{energyFLq} `directly.' 
 However, we find that consideration of the remaining case
  $|n|\sim |n_1|\sim |n_2| \sim |n_3|$ yields a restriction on $s$ given by \eqref{scond}.
We now describe this remaining case.
To simplify notation, we drop the frequency projections $\P_{\leq N}$. Furthermore, 
we let $\s=s-\frac 12 -\eps$ and we set
 $\tilde{w}_n=\jb{n}^{\s}w_{n}$ and ${\bf w}_n = \jb{n}^{s-\tilde{\eps}}w_n$.  
We employ the argument (and the notation) from subcase 8.2 in the proof of Proposition~\ref{Prop: Energy  Estimate}.
We have 
\begin{align*}
|\N_1(w)| & \les \sum_{n}\sum_{\rho\neq 0} \sum_{\G(n,\rho)} \frac{n_{\max}^{2s}}{|\rho| n_{\max}^{2\al-2}}  \frac{{\bf w}_{n_2}  {\bf w}_{n_3} \tilde{w}_{n}}{\jb{n}^{\s}\jb{n_1}^{\s-\tfrac 16}  \jb{n_2}^{s-\tilde{\eps}}\jb{n_3}^{s-\tilde{\eps}}} \bigg\| \sum_{\G(n)} \jb{m_1}^{\s-\frac 16} w_{m_1}\cj{w_{m_2}}w_{m_3} \bigg\|_{\l^{\infty}_{n_1}} \\
& \les \|w\|_{H^\s}^{3} \| w\|_{\FL^{s-\tilde{\eps},\infty}}^{2} \sum_{n}\sum_{\rho\neq 0} \sum_{\G(n,\rho)} \frac{1}{|\rho|n_{\max}^{\nu}}  \tilde{w}_{n},
\end{align*}
where 
\begin{align*}
\nu = 2\al+2s-\frac{19}{6}>0,
\end{align*}
provided $\al+s> \tfrac{19}{12}$.
Then, by Cauchy-Schwarz and the divisor counting lemma~\eqref{divcount}, we bound the above by
\begin{align*}
& \|w\|_{H^\s}^{4} \|w\|_{\mathcal{F}L^{s-\tilde{\eps},\infty}}^{2} \bigg( \sum_{n} \sum_{\G(n)} \frac{1}{|n-n_1|^{1-2\dl}|n-n_3|^{1-2\dl}n_{\max}^{2\nu}}     \bigg)^{\frac 12} \\
& \les  \|w\|_{H^{\s}}^{4} \|w\|_{\mathcal{F}L^{s-\tilde{\eps},\infty}}^2 \bigg( \sum_{n} \frac{1}{\jb{n}^{2\nu-6\dl}} \bigg)^{\frac{1}{2}}.
\end{align*}
Summing this requires $\nu>\tfrac{1}{2}$ which restricts us further to enforcing
\begin{align*}
\al + s > \frac{11}{6},
\end{align*}
completing the proof for $\mathcal{N}_1$.

Now, recall from the proof of Proposition~\ref{Prop: Energy Estimate} that we obtained the estimate 
\begin{align}
|\mathcal{R}_{1}(w)|\les \|\P_{\leq N}w\|_{H^{s-\frac 12 -\eps}}^{6}, \label{r1bdfl}
\end{align}
for $s>\max\Big(\tfrac 54 -\tfrac 12 \al, \tfrac 23 \Big)$ and $\al>\tfrac 12$. Since 
\begin{align*}
\max \bigg(  \frac{2}{3}, \frac{11}{6}-\al, \frac 54 -\frac 12 \al \bigg)=\max \bigg(  \frac{2}{3}, \frac{11}{6}-\al\bigg)
\end{align*}
for any $\al \in \R$, then we may simply use \eqref{FLembed} on two factors of \eqref{r1bdfl} to obtain \eqref{energyFLq} for $\mathcal{R}_{1}$. This completes the proof of \eqref{energyFLq}.

\end{proof}

\subsection{Construction of weighted Gaussian measures}\label{subsec:probmeas}

In this section, we construct weighted Gaussian measures which are adapted to the modified energy $E_{s}(v)$. Our attention is only on the low regularity setting $\frac 12 < s \leq 1$ and high enough dispersion ($\al \geq \tfrac 56 $), since the results in Section 4 established quasi-invariance when $s>1$ for any $\al>\tfrac 12$.

 Given $r>0$ and $N\geq 1$, we first wish to construct the measure 
\begin{align*}
d\rho_{s,N,r}(v)=Z_{s,N,r}^{-1}\ind_{ \{ \|v\|_{L^{2}}\leq r\}} e^{-\frac{1}{2}R_{s,N}(\P_{\leq N}v)}d\mu_{s}(v) 
\end{align*} 
and then, by taking $N\rightarrow \infty$, construct the measure
\begin{align*}
d\rho_{s,r}(v)=Z_{s,r}^{-1}\ind_{\{\|v\|_{L^{2}}\leq r\}} e^{-\frac{1}{2}R_{s}(v)}d\mu_{s}(v),
\end{align*}
where we recall
\begin{align*}
R_{s}(v)&= -\frac{1}{2}\text{Re} \sum_{\G(\cj n)} \frac{\psi_{s}(\cj n)}{ \phi(\cj n)} v_{n_1}\cj{v}_{n_2}v_{n_3}\cj{v}_{n_4}, 
\end{align*} 
and we define
\begin{align*}
R_{s,N}(v)&:= -\frac{1}{2}\text{Re} \sum_{\G_{N}(\cj n)} \frac{\psi_{s}(\cj n)}{ \phi(\cj n)} v_{n_1}\cj{v}_{n_2}v_{n_3}\cj{v}_{n_4}.
\end{align*}
We set 
\begin{align*}
F_{N,r}(v)= \ind_{ \{ \|v\|_{L^{2}}\leq r\}} e^{-\frac{1}{2}R_{s,N}(\P_{\leq N}v)} \,\, \,\,\text{and} \,\,\,\,
F_{r}(v)=\ind_{ \{ \|v\|_{L^{2}}\leq r\}} e^{-\frac{1}{2}R_{s}(v)}.
\end{align*}

The main result of this subsection is the following proposition which states, not only that the probability measure $\rho_{s,r}$ exists, but that we have `good' uniform $L^{p}$ bounds on the density for $\rho_{s,N,r}$ (see \eqref{unifNbd} below). Such higher $L^{p}$ bounds are crucial for the hybrid argument in \cite{hybrid} (see Lemma~\ref{lemma:flpdeviation} and Proposition~\ref{prop:diffeq}).

\begin{proposition}\label{prop:lpbound} Let $r>0$, $\al \geq \frac{3}{4}$ and $\max(\tfrac{5-4\al}{2}, \tfrac 12)< s\leq 1$.  Then, given $p<\infty$, there exists $C>0$ such that 
\begin{align}
\|F_{r}(v)\|_{L^{p}(\mu_s)}, \, \| F_{N,r}(v)\|_{L^{p}(\mu_s)} \leq C_{p,r,s,\al}, \label{unifNbd}
\end{align}
uniformly in $N\in \mathbb{N}$. Furthermore, there exists $R_s(v)\in L^{p}(\mu_s)$ such that 
\begin{align}
\lim_{N\rightarrow \infty} R_{s,N}(\P_{\leq N}v)=R_{s}(v) \qquad \text{in} \,\,\,\, L^{p}(\mu_s) \label{LpconvR}
\end{align}
and 
\begin{align}
\lim_{N\rightarrow \infty} F_{N,r}(v)=F_{r}(v)\qquad \text{in} \,\,\,\, L^{p}(\mu_s). \label{LpconvF}
\end{align}

\end{proposition}

In order to prove Proposition~\ref{prop:lpbound} by employing the argument in \cite[Proposition 6.2]{oh2016quasi}, we need the following bound. Note that we define $R_{s,\infty}(v)=R_{s}(v)$.
\begin{lemma}\label{lemma:Rsnbd}
Let $\al > \frac 12$ and $\frac{1}{2}<s\leq 1$. Then for any 
\begin{align*}
\g> \max\bigg(0, \frac{2s+1-2\al}{3},\frac{1}{4}+s-\al \bigg), 
\end{align*} we have
\begin{align}
|R_{s,N}(\P_{\leq N}v)| \les \|\P_{\leq N}v\|_{L^{2}}\| \P_{\leq N}v\|_{H^{\g}}^{3}, \label{Rbound}
\end{align} uniformly in $N\in \mathbb{N}\cup \{ \infty\}$. In particular, if $\al>\frac 32$, we may take $\g\equiv 0$ in \eqref{Rbound}.
\end{lemma}

\begin{proof}

Notice that we have a symmetry with respect to the interchange of $n_1$ and $n_3$ and the interchange of $n_2$ and $n$. 
We split the proof of \eqref{Rbound} into a few cases with the remaining cases following analogously by exploiting this symmetry.  Below we prove \eqref{Rbound} for $N=\infty$ as it is clear how to adjust the argument when $N\in\mathbb{N}$.
We write $\tilde{v}_{n}:=\jb{n}^{\g} v_{n}$.

\medskip

\noi
$\bullet$ \textbf{Case 1:} $|n_1|\sim|n| \gg |n_2|, |n_3|$

In this case, it is clear from Lemma~\ref{lemma: phase lower bound} that 
\begin{align*}
|\phi(\cj n)| \gtrsim n_{\max}^{2\al-1}|n-n_1|,
\end{align*}
and from the mean value theorem,
\begin{align*}
|\psi(\cj n)|\les n_{\max}^{2s-1}|n-n_1|.
\end{align*}
Hence by Cauchy-Schwarz,
\begin{align*}
|R_{s}(v)| & \les \bigg( \sum_{\G(\cj n)}\frac{\tilde{v}_{n_1}^{2}}{ n_{\max}^{2(2\al-2s+2\g)}\jb{n_3}^{2\g}}\bigg)^{\frac{1}{2}} \| v\|_{H^{\g}}^{2}\|v\|_{L^{2}}  \les  \|v\|_{L^{2}}\| v\|_{H^{\g}}^{3},
\end{align*}
provided 
\begin{align}
2(2\al-2s+2\g)>1 \qquad \text{and} \qquad   2\al-2s+3\g>1. \label{gammacond}
\end{align}

\medskip

\noi
$\bullet$ \textbf{Case 2:} $|n_2|\sim|n| \gg |n_1|, |n_3|$

Using $|\phi(\cj n)|\gtrsim n_{\max}^{2\al}$ and applying Cauchy-Schwarz as in the previous case, we obtain \eqref{Rbound} provided $\g$ satisfies \eqref{gammacond}.

\medskip

\noi
$\bullet$ \textbf{Case 3:} $|n_1|\sim|n_2|\sim |n_3|\sim |n|$

On $\G(\cj n)$, we have
\begin{align*}
|\psi_{s}(\cj n)| \les |n-n_3|n_{\max}^{2s-1} \qquad \text{and} \qquad |\phi(\cj n)| \ges |n-n_3||n-n_1|n_{\max}^{2\al-2}.
\end{align*}
With $\g\geq 0$ to be determined, we have 
\begin{align*}
|R_{s}(v)| & \les \sum_{\G(\cj n)} \frac{n_{\max}^{2s-1}}{|n-n_1|n_{\max}^{2\al -2}} \frac{\tilde{v}_{n_1} v_{n_2} \tilde{v}_{n_3}\tilde{v}_{n}}{\jb{n_1}^{\g}\jb{n_3}^{\g}\jb{n}^{\g}} \\
& \les \sum_{\G(\cj n)} \frac{1}{|n-n_1| n_{\max}^{\nu}} \tilde{v}_{n_1} v_{n_2}v_{n_3}\tilde{v}_{n},
\end{align*}
where $\nu = 2\al-1-2s+3\g>0$
provided 
\begin{align*}
\g>\max \bigg(0, \frac{2s+1-2\al}{3} \bigg).
\end{align*}
By the Cauchy-Schwarz inequality and Lemma~\ref{lemma:sumestimate}, we have 
\begin{align*}
& \les \bigg( \sum_{n, n_1,n_3}\frac{ \tilde{v}_{n}^{2}\tilde{v}_{n_3}^{2} }{\jb{n-n_1}^{1+\dl} }\bigg)^{\frac 12}\bigg( \sum_{n_2,n_1,n} \frac{v_{n_2}^{2}\tilde{v}_{n_1}^{2}}{ \jb{n-n_3}^{1-\dl}\jb{n}^{2\dl}} \bigg)^{\frac 12} \\
& \les  \|v\|_{L^{2}}\| v\|_{H^{\g}}^{3}.
 \end{align*}

Notice from the condition $2\al-1-2s+2\g>0$, that if $\al >\frac{3}{2}$, we can take $\g=0$. This completes the proof. 
\end{proof}

We also require the following probabilistic estimate, see \cite[Lemma 6.4]{oh2016quasi}.

\begin{lemma}\label{LEM:g}
Let $\{ g_{n}\}_{n\in \Z}$ be independent standard complex-valued Gaussian random variables. Then, there exist $c,C>0$ such that, for any $M\geq 1$, we have
\begin{align*}
\mathbb{P}\bigg[ \bigg(  \sum_{n=1}^{M}|g_n |^{2} \bigg)^{\frac 12} \geq K \bigg] \leq e^{-cK^2},
\end{align*}
provided $K\geq CM^{\frac 12}$.
\end{lemma}

We now give the proof of Proposition~\ref{prop:lpbound}.

\begin{proof}[Proof of Proposition~\ref{prop:lpbound}]
For $\al>\tfrac 32$, Lemma~\ref{lemma:Rsnbd} implies we may take $\g\equiv 0$ in \eqref{Rbound} and hence
\begin{align*}
 \ind_{\{\|v\|_{L^2}\leq r\}} |R_{s,N}(\P_{\leq N}v)| \les \ind_{\{\|v\|_{L^2}\leq r\}}\|\P_{\leq N}v\|_{L^{2}}^{4}\les r^{4}, 
 \end{align*} 
 at which point, the bound \eqref{unifNbd} follows trivially.
We make up the remaining case $\frac{3}{4}\leq \al \leq \frac 32$ in the following.
Given $\frac{1}{2}<s\leq 1$, let $\g$ be as in Lemma~\ref{lemma:Rsnbd} whose precise value will be specified later.
 On $\{ \|v\|_{L^{2}}\leq r\}$, \eqref{Rbound} implies 
\begin{align*}
|R_{s}(\P_{\leq M_{0}}v)|\leq C_{0}r\|\P_{\leq M_0}v\|_{H^{\g}}^{3} \leq C_{0}M_{0}^{3\g}r^{4}.
\end{align*}
We have 
\begin{align}
\|F_{r}(v)\|_{L^{p}(d\mu_s)}^{p}& \leq C^{p} +p\int_{\max(e,e^{2^{\frac 32}C_0 r})}^{\infty} \ld^{p-1}\mu_{s}( |R_{s}(v)|\geq \log \ld, \,\,\, \|v\|_{L^2}<r )  d\ld. \label{lc1}
\end{align}
We choose $M_0 >0$ such that 
\begin{align}
\log \ld = 2^{\frac 32} C_0 M_0^{3\g}r^{4}. \label{logld}
\end{align}
For $j\in \mathbb{N}$, let $M_j =2^{j} M_0$ and $\s_j=C_{\eps}2^{-\eps j} = CM_{0}^{\eps}M_{j}^{-\eps}$ for some small $\eps>0$ such that $\sum_{j=1}^{\infty} \s_{j} =\frac{1}{2}$. 
Then we have 
\begin{align*}
\mu_{s}( |R_{s}(v)|\geq \log \ld, \,\,\, \|v\|_{L^2}<r ) & \leq \mu_{s}( \|v\|_{H^{\g}}^{2}\geq (C_{0}^{-1}r^{-1}\log \ld)^{\frac 23} ) \\
& \leq \sum_{j=1}^{\infty} \mu_{s}( \|\P_{M_j}v\|_{H^{\g}}^{2} \geq \s_{j}(C_{0}^{-1}r^{-1}\log \ld)^{\frac 23}) \\
& \les \sum_{j=1}^{\infty} \mathbb{P} \bigg(  \bigg(\sum_{|n|\sim M_j} |g_n|^{2} \bigg)^{\frac 12} \gtrsim L_{j}  \bigg),
\end{align*}
where $L_{j}:=(C_{0}^{-1}r^{-1}\log \ld )^{\frac 13}\s_{j}^{\frac{1}{2}} M_j^{s-\g}\gtrsim M_{0}^{\frac 12 \eps}M_{j}^{s-\g-\frac{1}{2}\eps}\gg M_{j}^{\frac{1}{2}}$, provided $s-\g>\frac{1}{2}$. 
We used here that $\ld> e^{2^{\frac{3}{2}}C_0 r}$ implies, from \eqref{logld}, $M_{0}^{\g}r \geq 1$ and hence $(C_0^{-1}r^{-1}\log \ld)^{2/3} \sim M_{0}^{4\g}r^4 \gtrsim 1$.
Therefore, by Lemma~\ref{LEM:g}, we have
\begin{align*}
\mu_{s}( |R_{s}(v)|\geq \log \ld, \,\,\, \|v\|_{L^2}<r ) & \les \sum_{j=1}^{\infty} e^{-c_{r}2^{j(2s-2\g -\frac{2}{3}\eps)} (\log \ld)^{ \frac{2}{3}+\frac{2}{3}\frac{ s-\g}{\g}  }}\\
& \les e^{-c''_{r} (\log \ld)^{ \frac{2s}{3\g} }}.
\end{align*}
Thus, from \eqref{lc1}, we have 
\begin{align*}
\|F_{r}(v)\|_{L^{p}(d\mu_s)}^{p}& \les C_{p}+p\int_{C}^{\infty} e^{p\ld}e^{-c_{r}'' \ld^{\frac{2s}{3\g}}}d\ld <C<\infty,
\end{align*}
provided $\frac{2}{3}s>\g$. It is clear that the above arguments also apply to obtain the uniform bound \eqref{unifNbd} when $N\in \mathbb{N}$.

Thus we can construct the measure $\rho_{s,r}$ provided we can choose $\g\in \R$ satisfying
\begin{align*}
\max\bigg( 0, \frac{2s+1-2\al}{3}, \frac 14 +s-\al \bigg) <\g <\min\bigg( s-\frac{1}{2}, \frac{2}{3}s\bigg)=s-\frac 12.
\end{align*} 
As we wish to consider $s$ close to $1$, we must impose $\al>\tfrac 34$ to rule out the maximum on the left hand side being $\tfrac{1}{4}+s-\al$.
Now, if $s+\frac 12 -\al \leq 0$, it is clear we can pick a $\g>0$. 
Otherwise, if $s+\frac{1}{2}-\al\geq 0$, we can choose $\g>0$ as long as 
\begin{align*}
\frac{2s+1-2\al}{3}<s-\frac 12,
\end{align*}
which upon rearranging yields the condition $\max(\tfrac{5-4\al}{2}, \tfrac 12)< s\leq 1$.

As for the $L^{p}(\mu_{s})$ convergence of $R_{s,N}(\P_{\leq N} v)$ and $F_{N,r}(v)$, we note that when $\al >\frac{3}{2}$, we have 
\begin{align}
|R_{s,N}(\P_{\leq N}(v))-R_{s}(v)| \les \|\P_{>N}v\|_{L^{2}}\|v\|_{L^{2}}^{3}. \label{conv61}
\end{align}
By a slight modification of \eqref{Rbound}, when $\tfrac 12 <\al <\tfrac 32$, we also have 
\begin{align}
|R_{s,N}(\P_{\leq N}(v))-R_{s}(v)| \les \|\P_{>N}v\|_{L^{2}}\|v\|_{H^{\g}}^{3}+\|v\|_{L^2}\|v\|_{H^\g}^2 \|\P_{>N}v\|_{H^\g}. \label{conv62}
\end{align}
Taking $N\rightarrow \infty$ in \eqref{conv61} and \eqref{conv62} and noting that $s-\g>\tfrac{1}{2}$ shows $R_{s,N}(\P_{\leq N}v)$ converges almost surely with respect to $\mu_{s}$ to $R_{s}(v)$. Then because of the uniform in $N$ bounds 
\begin{align*}
\| R_{s,N}(\P_{\leq N}v)\|_{L^{p}(\mu_s)},\|R_{s}(v)\|_{L^{p}(\mu_s)}\leq C_{p,s}<\infty,
\end{align*}
which follow from Lemma~\ref{lemma:Rsnbd}, a standard argument using Egoroff's theorem implies \eqref{LpconvR} (see \cite[Proposition 6.2]{oh2016quasi}).

From \eqref{conv61} and \eqref{conv62}, we have almost sure convergence of $F_{r,N}(v)$ to $F_{r}(v)$ with respect to $\mu_{s}$. Using \eqref{unifNbd}, the above standard argument implies convergence in $L^p(\mu_s)$; namely \eqref{LpconvF}.
This completes the proof of Proposition~\ref{prop:lpbound}.

\end{proof}

The next lemma shows that the two factors lying in $\FL^{\s,\infty}$ for $\s<s$ in the modified energy estimate \eqref{energyFLq} have moments indeed contributing a factor of $p^{\frac 12}$. Notice that as a consequence
of Proposition~\ref{prop:lpbound} (namely, $Z_{s,N,r}\rightarrow Z_{s,r}$ as $N\rightarrow \infty$), $Z_{s,N,r}^{-1}$ is bounded uniformly in $N\in \mathbb{N}$.

\begin{lemma} \label{lemma:flpdeviation}
Given $\eps>0$ and $r>0$, there exists $C=C(\eps,r)>0$ such that 
\begin{align*}
\big\|  \| f \|_{\FL^{s-\eps,\infty}} \big\|_{L^{p}(\rho_{s,N,r})} \leq Cp^{\frac 12}
\end{align*}
for any $p\geq 1$ and $N\in \mathbb{N}$.
\end{lemma}
\begin{proof}
Applying the uniform bound \eqref{unifNbd}, the uniform bound on $Z_{s,N,r}^{-1}$ and  Minkowski's integral inequality, for any $q>\tfrac{1}{\eps}$, we have
\begin{align*}
\big\|  \|f\|_{\FL^{s -\eps,\infty}} \big\|_{L^{p}(\rho_{s,N,r})} & \leq \big\|  \|f\|_{\FL^{s -\eps,q}} \big\|_{L^{p}(\rho_{s,N,r})}\\
& \leq Z_{s,N,r}^{-\frac 1p}\| F_{N,r} \|_{L^{q'}(\mu_s)}^{\frac 1p} \big\| \|f\|_{\FL^{s -\eps,q}} \big\|_{L^{pq}(\mu_s)}^{\frac{1}{p}} \\
& \les_{q} \big\| \jb{n}^{-\eps} \|g_{n}\|_{L^{pq}(\O)} \big\|_{\l^{q}_n} \\
& \les_{q} q^{\frac 12}p^{\frac 12},
\end{align*} where we have used the following well-known estimate on higher moments of Gaussian random variables in the last inequality: 
\begin{align}
\|g_{n}\|_{L^{p}(\O)}\les p^{\frac{1}{2}} \label{LPgauss}
\end{align}
for any $p\geq 2$.
\end{proof}

\subsection{Transport of the truncated weighted Gaussian measures}

In this subsection, we study how the measures $\rho_{s,N,r}$ evolve under the flow of the truncated equation \eqref{Truncated equation for v}. We follow the method of \cite{hybrid} in which we use a `change of variables formula' (see Lemma~\ref{lemma:cov}) to make the modified energy $E_{s}$ along the truncated flow appear. 
Taking a time derivative and using the estimate \eqref{energyFLq} then gives a differential inequality for the evolution of $\rho_{s,N,r}$ under $\Phi_{N}$ (see Proposition~\ref{prop:diffeq}).

\begin{lemma}[Change of variables formula]  \label{lemma:cov}
Let $\al \geq \frac{5}{6}$, $s$ be as in \eqref{scond} and $r>0$. 
Then for any $N\in \mathbb{N}$, $t\in \R$ and measurable set $A \subset  H^{s-\frac 12-\eps}(\T)$, we have 
\begin{align}
\rho_{s,N,r}(\Phi_{N}(t)(A))={\ft Z}^{-1}_{s,N}\int_{A} \ind_{\{\|v\|_{L^2}\leq r\}}\, e^{-E_{s,N}(\P_{\leq N}\Phi_{N}(t)(v))} du_{\leq N}\times d\mu_{s,N}^{\perp}. \label{changeofvar}
\end{align}

\end{lemma}

We omit the proof of Lemma~\ref{lemma:cov} as it is identical to those in \cite{hybrid, oh2016quasi, oh2017quasi}. The core ingredients are the invariance of the Lebesgue measure $L_{N}$ under the truncated flow $\Phi_{N}$ (because of Liouville's theorem), the invariance of $\Phi_{N}$ in the $L^{2}$-norm (mass conservation) and the bijectivity of the flow $\Phi_{N}$. When $\al\geq 1$, \eqref{changeofvar} also holds for any measurable $A \subset L^{2}(\T)$.

\begin{proposition}\label{prop:diffeq}
Let $\al \geq \frac{5}{6}$ and $s$ be as in \eqref{scond}. Then, given $r,R>0$ and $T>0$, there exists $C_{r,R,T}>0$ such that 
\begin{align}
\frac{d}{dt}\rho_{s,N,r}(\Phi_{N}(t)(A))\leq C_{r,R,T} \cdot p \,  \{ \rho_{s,N,r}(\Phi_{N}(t)(A)) \}^{1-\frac 1p} \label{diffeq1}
\end{align}
for any $p\geq 2$, any $N\in \mathbb{N}$, any $t\in [0,T]$ and any measurable set $A\subset B_{R}\subset  H^{\s}(\T)$.
\end{proposition}

\begin{proof} 
Fix $r,R>0$ and $T>0$.
As a preliminary step, we first note the following key estimate on the growth of the modified energy $E_{s,N}$: for $\al$ and $s$ as in Proposition~\ref{prop:diffeq}, we have
\begin{align}
\bigg\| \ind_{B_{R}}  \dt E_{s,N}(\P_{\leq N}\Phi_{N}(t)(v)) \Big\vert_{t=0}   \bigg\|_{L^{p}(\rho_{s,N,r})} \leq C_{r,R}\cdot \, p \label{energyp}
\end{align}
for any $p\geq 2$ and for any $N\in \mathbb{N}$.
This follows from \eqref{energyFLq}, Lemma~\ref{lemma:flpdeviation}, the uniform bound \eqref{unifNbd} on $F_{N,r}$, the uniform bound on $Z_{s,N,r}^{-1}$ and Cauchy-Schwarz inequality, since
\begin{align*}
\text{LHS of} \,\, \eqref{energyp}  &\leq Z_{s,N,r}^{-\frac 1p} \|F_{N,r}(v)\|_{L^{2}(\mu_s)}^{\frac 1p}\|\ind_{B_R}  \dt E_{s,N}(\P_{\leq N}\Phi_{N}(t)(v)) \vert_{t=0} \|_{L^{2p}(\mu_{s})} \\
& \leq C \big\| \ind_{B_R} \|\P_{\leq N} v\|_{\FL^{s-\frac{\eps}{2},\infty}}^{2}\|\P_{\leq N}v\|_{H^{\s}}^{4}\big\|_{L^{2p}(\mu_s)} \\
& \leq CR^{4} \big\|  \|\P_{\leq N} v\|_{\FL^{s-\frac{\eps}{2},\infty}} \big\|_{L^{4p}(\mu_s)}^{2} \\
& \leq CR^{4} p.
\end{align*}

Now fix a measurable set $A\subset B_{R}\subset H^{\s}(\T)$ and $t_{0}\in [0,T]$. By the semigroup property of $\Phi_{N}(t)$ and the change of variables formula (Lemma~\ref{lemma:cov}), we have
\begin{align*}
\frac{d}{dt}\rho_{s,N,r}&(\Phi_N (t)(A))\bigg\vert_{t=t_0}=Z^{-1}_{s,N,r}\frac{d}{dt}\int_{\Phi_{N}(t)(A)} \ind_{\{ \|v\|_{L^2} \leq r\}} e^{-R_{s,N}(\P_{\leq N}v)}d\mu_{s}(v) \bigg\vert_{t=t_{0}} \\
& = Z_{s,N,r}^{-1}\frac{d}{dt}\int_{\Phi_N (t)(\Phi_N (t_0)(A))} \ind_{\{ \|v\|_{L^2} \leq r\}} e^{-R_{s,N}(\P_{\leq N}v)}d\mu_{s}(v) \bigg\vert_{t=0} \\
& = \ft{Z}^{-1}_{s,N,r}\frac{d}{dt}\int_{\Phi_{N}(t_0)(A)} \ind_{\{ \|v\|_{L^2} \leq r\}} e^{-E_{s,N}(\P_{\leq N}\Phi_{N}(t)(v))}du_{\leq N} \times d\mu_{s,N}^{\perp} \bigg\vert_{t=0} \\
&=-Z^{-1}_{s,N,r}\int_{\Phi_{N}(t_0)(A)} \ind_{\{ \|v\|_{L^2} \leq r\}}  \dt E_{s,N}(\P_{\leq N}\Phi_{N}(t)(v))\vert_{t=0}\, e^{-R_{s,N}(\P_{\leq N}v)}d\mu_{s}(v).
\end{align*}
Now recall from Proposition~\ref{Prop: GWP growth bound}~(i) that for any $t\in [0,T]$ and $N \in \mathbb{N}$, there exists $C(R,T)>0$ such that $\Phi_{N}(t)(B_R)\subset B_{C(R,T)}$. 
Note that when the flow $\Phi$ is only well-defined locally-in-time, we use Proposition~\ref{Prop: GWP growth bound}~(i).
Hence, by H\"{o}lder's inequality we obtain 
\begin{align*}
\frac{d}{dt}\rho_{s,N,r}&(\Phi_N (t)(A))\bigg\vert_{t=t_0}  \\
& \leq Z_{s,N,r}^{-1}\int_{\Phi_{N}(t_0)(A)} \big\vert \dt E_{s,N}(\P_{\leq N}\Phi_{N}(t)(u))|_{t=0} \big\vert \ind_{\{ \|v\|_{L^2} \leq r\}} e^{-R_{s,N}(\P_{\leq N}v)}d\mu_{s}(v) \\
& \leq \bigg\| \ind_{B_{C(R,T)}}  \dt E_{s,N}(\P_{\leq N}\Phi_{N}(t)(v)) \Big\vert_{t=0}   \bigg\|_{L^{p}(\rho_{s,N,r})} \{ \rho_{s,N,r}(\Phi_{N}(t_0)(A)) \}^{1-\frac 1p}.
\end{align*}
Applying \eqref{energyp} yields \eqref{diffeq1}. 

\end{proof}

 \subsection{Proof of Theorem~\ref{Main result} \eqref{resultals} }

In this section, we apply the argument in \cite{oh2017quasi} to deduce from Proposition~\ref{prop:diffeq}, the quasi-invariance of $\mu_{s}$ under the untruncated flow $\Phi(t)$. In what follows, we fix $\al \geq \tfrac{5}{6}$ and consider $s$ satisfying \eqref{scond}. 
We show that for each fixed $R>0$,
\begin{align}
\text{if} \,\,\,\, \mu_{s}(A)=0, \qquad \text{then}\qquad\,\,\,\, \mu_{s}(\Phi(t)(A))=0  \label{54}
\end{align}
  for any $t\in [0,T(R)]$ and for any measurable set $A\subset B_{R}$. 
This implies local-in-time quasi-invariance of $\mu_{s}$ under \eqref{FNLS} for any
\begin{align*}
s>\frac{11}{6}-\al.
\end{align*} 
When $\al\geq 1$, \eqref{54} is true for any $t\in \R$. As $R$ is arbitrary, this implies quasi-invariance of $\mu_{s}$ under the dynamics of FNLS~\eqref{FNLS}.
For the rest of this section, we fix $R>0$ and $\al \geq 1$ since the arguments below are easily modified to imply local-in-time quasi-invariance when $\tfrac 56<\al<1$.

For the first step, we use Proposition~\ref{prop:diffeq} to show that $\rho_{s,N,r}$ is quasi-invariant under $\Phi_{N}(t)$; see Lemma~\ref{rhosnrqi}. The proof of Lemma~\ref{rhosnrqi} follows exactly as in \cite[Proposition 5.3]{oh2017quasi}.

\begin{lemma}\label{rhosnrqi}
Given $r>0$, there exists $0<t_{r,R}<T$ such that given $\eps>0$, there exists $\dl>0$ such that if, for a measurable set $A\subset B_{R}\subset H^{\s}(\T)$, there exists $N_{0}\in \mathbb{N}$ such that 
\begin{align*}
\rho_{s,N,r}(A)<\dl
\end{align*}
for any $N\geq N_0$, then we have 
\begin{align*}
\rho_{s,N,r}(\Phi_{N}(t)(A))<\eps
\end{align*}
for any $t\in [0,t_{r,R}]$ and any $N\geq N_{0}$.
\end{lemma}

Then a careful argument allows the previous statement to hold when $N=\infty$; that is, we have that $\rho_{s,r}$ is quasi-invariant under the untruncated flow $\Phi(t)$ (Lemma~\ref{rhosrqi}). The proof of Lemma~\ref{rhosrqi} makes use of the approximation property of the dynamics of FNLS~\eqref{FNLS} as in Proposition~\ref{Prop: GWP growth bound}~(ii) and follows the arguments in \cite[Lemma 5.5]{oh2017quasi}. 

\begin{lemma}\label{rhosrqi}
Given $r>0$, there exists $0<t_{r,R}<T$ such that given $\eps>0$, there exists $\dl>0$ such that if 
\begin{align*}
\rho_{s,r}(A)<\dl,
\end{align*} then we have 
\begin{align*}
\rho_{s,r}(\Phi_{N}(t)(A))<\eps
\end{align*}
for any $t\in [0,t_{r,R}]$.
\end{lemma}

Now, invoking the mutual absolute continuity of $\rho_{s,r}$ and $\mu_{s,r}$ implies $\mu_{s,r}$ is quasi-invariant under $\Phi(t)$. We then take $r\rightarrow \infty$ (as in \cite[Theorem 1.2]{oh2017quasi}) and iterate in time to obtain \eqref{54} for every $t\in \R$, for a fixed $R>0$. This concludes the proof of Theorem~\ref{Main result}~(i).

\begin{appendix}

\section{Proof of Lemma \ref{lemma: phase lower bound} for $\al>1$} \label{app:lowerbd}
 
Setting $k=n_{1}-n$ and $j=n_{3}-n$, it is equivalent to prove
 \begin{align*}
g(j,k,n):=| |n+k|^{2\alpha}-|n+k+j|^{2\alpha}+|n+j|^{2\alpha}-|n|^{2\alpha}| \gtrsim |k||j|(|k|+|j|+|n|)^{2\alpha-2}. 
\end{align*}
Since $2\alpha>2$, the function $f(x)=|x|^{2\alpha}\in C^{2}(\R)$ 
and satisfies $$f'(x)=2\alpha |x|^{2\alpha -2}x, \quad f''(x)=2\alpha (2\alpha -1)|x|^{2\alpha -2}.$$ 
We follow a similar argument to the case $\frac 12<\al<1$ in \cite{demirbas2013existence}. 
Without loss of generality we can assume that $\max(|j|,|k|)=|j|$ and $j\neq 0$.
 For any $c\in \R$, define $f_{c}(x):=|x+c|^{2\alpha}-|x-c|^{2\alpha}.$
 Then, we have $$g(j,k,n)=|f_{j/2}(n+j/2)-f_{j/2}(n+k+j/2)|.$$
The mean value theorem implies $$g(j,k,n)\gtrsim                                                                                                                                                                                                                                                                                           |k|\min_{x\in I}|f_{j/2}'(x)|,$$ where $I$ is either the interval $(n+j/2,n+j/2+k)$ or the interval 
$(                                                                                                                                                                                                                                                                                                                                                                                                                                                                                                                                                                                                                                                                                  n+j/2+k,n+j/2)$. 
It suffices to show 
\begin{equation} 
|f_{c}'(x)|\gtrsim |c| \max( |x|, |c|)^{2\al-2} \textnormal{ for } |c|\geq \frac 12. \label{lowerderiv}
\end{equation} 
To see this, we first suppose $|n|\les |j|$. Then, for any $x\in I$, we have
$$ |f'_{j/2}(x)|\gtrsim   |j|^{2\alpha-1}\gtrsim |j|(|k|+|j|+|n|)^{2\alpha-2}. $$
Now suppose $|n|\gg |j|$. Then $x\in I$ implies $|x|\sim |n|$ and hence 
\begin{align*}
\min_{x\in I} |f'_{j/2}(x)| \gtrsim |j| |n|^{2\al-2}\gtrsim |j|(|k|+|j|+|n|)^{2\al-2}.
\end{align*}
In order to verify \eqref{lowerderiv}, we may assume that $x\geq 0$ as $f_{c}$ is odd and similarly, we assume $c\geq \frac 12$ as $f'_{c}$ is odd in $c$. We have $$f_{c}'(x)=2\alpha|x+c|^{2\alpha-2}(x+c)-2\alpha |x-c|^{2\alpha-2}(x-c),$$ 
and we consider three cases. 
\medskip

\noindent
$\bullet$\textbf{Subcase 2.1:} $0\leq x\leq c$ 

Here we have $$f_{c}'(x)=2\alpha \left[|x+c|^{2\alpha-1}+|x-c|^{2\alpha-1} \right] \gtrsim c^{2\alpha-1}.$$

\medskip

\noindent
$\bullet$\textbf{Subcase 2.2:} $c<x\leq 2c$

We have
\begin{align*}
f_{c}'(x)& =2\alpha \left[(x+c)^{2\alpha-1}-(x-c)^{2\alpha-1} \right] \\
& = 2\alpha c^{2\alpha-1}\left[ \left(\frac{x}{c}+1 \right)^{2\alpha-1}-\left(\frac{x}{c}-1 \right)^{2\alpha-1} \right]  \gtrsim c^{2\alpha-1} \left( \frac{x}{c}\right)^{2\alpha-2}  \sim  cx^{2\al-2}.
\end{align*}

\medskip

\noindent
$\bullet$ \textbf{Subcase 2.3:} $ x>2c$

Using the mean value theorem, we have \begin{align*}
f_{c}'(x)& =2\alpha x^{2\alpha-1}\left[\left(1+\frac{c}{x} \right)^{2\alpha-1}-\left(1-\frac{c}{x} \right)^{2\alpha-1} \right]  \gtrsim x^{2\alpha-1}\frac{c}{x} \sim  cx^{2\al-2}.
\end{align*}
This completes the proof of \eqref{lowerderiv}.

\section{Well-posedness theory for FNLS with $\al>1$ in $L^{2}(\T)$}\label{app:gwp}

In this appendix, we detail the global well-posedness of the Cauchy problem:
\begin{equation}\label{FNLS1}
\begin{cases}
i\partial_t u+(-\dx^2)^\alpha u=\pm |u|^2u\\
u|_{t=0}=u_0\in H^{s}(\T),
\end{cases}
\end{equation}
with $\al>1$ and for any $s\geq 0$. 
We say $u\in C([0,T]; H^s(\T))$ is a solution to \eqref{FNLS1} if it satisfies the following integral (Duhamel) formulation
\begin{align*}
u(t)= S(t)u_0 \mp  i\int_{0}^{t}S(t-t')|u(t')|^2 u(t')dt',
\end{align*}
where $S(t)=e^{-it(-\dx^2)^\al}$.
The main result is the following:

\begin{proposition} \label{prop:appBlocal} Let $s\geq 0$ and $\al >1$. Then, given $u_0\in H^s(\T)$, there exist $T=T(\|u_0\|_{L^2})>0$ and a unique solution $u\in C([-T,T]; H^{s}(\T))$ to \eqref{FNLS1} with $u|_{t=0}=u_0$. Furthermore, we have 
\begin{align*}
\sup_{t\in [-T,T]} \|u(t)\|_{H^{s}} \leq C\|u_0\|_{H^{s}}.
\end{align*}
\end{proposition}

While the overall arguments presented here are standard, to our knowledge, they have not been fully explicated in the literature for \eqref{FNLS1} apart from when $\al=2$, see \cite[Appendix A]{oh2016quasi}. We follow the ideas presented in \cite{demirbas2013existence}, which considered $\tfrac{1}{2}<\al<1$.
The crucial ingredient is the lower bound on the phase function $\phi(\cj n)$ given in Lemma~\ref{lemma: phase lower bound}. 
As the global well-posedness for $s>0$ will follow by iterating the local argument using the mass conservation (see \eqref{mass}), and the local argument is independent of the defocusing or focusing nature of \eqref{FNLS1}, we may assume \eqref{FNLS1} is defocusing (that is, the sign on the nonlinearity in \eqref{FNLS1} is positive).

We will show the well-posedness of the gauged equation
\begin{align}
i\dt v + (-\dx^2)^{\al} v =\bigg( |v|^{2} -2\int_{\T} |v|^2 dx\bigg)v, \label{gauged}
\end{align}
in $H^{s}(\T)$ for $s\geq 0$. This suffices from the fact that given the solution $v\in C(\R; L^{2}(\T))$ satisfying $v|_{t=0}=u_0$ to \eqref{gauged}, the function
\begin{align*}
u(t)=\mathcal{G}^{-1}[v](t)=e^{-2it \fint |v|^{2} dx}v
\end{align*}
is the solution to \eqref{FNLS1} with $u|_{t=0}=u_0$ (see Section~\ref{section:gauge}).
We define the non-resonant and resonant operators (respectively)  by
\begin{align}
\begin{split}
\mathcal{N}(v_1,v_2,v_3)&=\sum_{n\in \Z} e^{inx} \sum_{\G(n)} \ft{v_1}(n_1)\cj{\ft{v_2}(n_2)}\ft{v_3}(n_3) \,\,\text{and} \\
 \mathcal{R}(v_1,v_2,v_3) &= -\sum_{n\in \Z}e^{inx} \ft{v_1}(n) \cj{\ft{v_2}(n)}\ft{v_3}(n), 
\end{split} \label{nonlinops}
\end{align}
where for fixed $n\in \Z$, the set $\G(n)$ is the non-resonant hyperplane given in \eqref{nrplane}. We will write $\mathcal{N}(v,v,v)=\mathcal{N}(v)$ and the same for $\mathcal{R}$. 
Then, the integral formulation of \eqref{gauged} becomes 
\begin{align}
v(t)=S(t)u_{0}-i\int_0^t S(t-t')( \mathcal{N}(v)(t') +\mathcal{R}(v)(t')) dt'=: \Lambda(v)(t).  \label{Lambda}
\end{align}

Our goal is to prove local well-posedness of \eqref{gauged} via a contraction mapping argument for the operator $\Lambda$ in the Fourier-restriction spaces $X^{s,b}(\R\times \T)$. We now state some basic properties of the spaces $X^{s,b}(\R\times \T)$.
Given $s,b\in \R$, we define the space $X^{s,b}(\R\times \T)$ via the norm 
\begin{align*}
\|v\|_{X^{s,b}(\R\times \T)}=\| \jb{n}^{s}\jb{\tau -|n|^{2\al} }^{b}\, \ft{v}(\tau, n)\|_{L^{2}_{\tau}\l^{2}_{n}(\R\times \Z)},
\end{align*}
where $\ft{v}(\tau, n)$ denotes the space-time Fourier transform of $v(t,x)$.
 Given $T>0$, we also define the local-in-time version $X^{s,b}([0,T]\times \T)$ of $X^{s,b}(\R\times \T)$ as 
\begin{align*}
X^{s,b}([0,T]\times \T)=\inf\{ \|v\|_{X^{s,b}(\R\times \T)} \, : \, v|_{[0,T]}=u \}.
\end{align*}
We will denote by $X^{s,b}$ and $X^{s,b}_{T}$ the spaces $X^{s,b}(\R\times \T)$ and $X^{s,b}([0,T]\times \T)$, respectively. 
We have the following embedding: for any $s\in \R$ and $b>\frac{1}{2}$, we have
\begin{align}
X^{s,b}_T \hookrightarrow C([0,T];H^{s}(\T)). \label{ctsembed}
\end{align}
Given any function $F$ on $[0, T]\times \T$, we denote by $\tilde{F}$ any extension of $F$ onto $\R\times \T$.

The following linear estimates hold in $X^{s,b}_{T}$. Their proofs are standard and can be found in, for example, \cite{GTV}.

\begin{lemma}
The following are true:
\begin{enumerate}[\normalfont (i)]
\setlength\itemsep{0.3em}
\item \emph{ [Homogeneous linear estimate]} Given $s,b\in \R$, we have 
\begin{align}
\| S(t)v\|_{X^{s,b}_T} \les \|v\|_{H^s}, \label{linest}
\end{align}
for any $0<T\leq 1$. 
\item \emph{ [Nonhomogeneous linear estimate]} Let $s\in \R$ and $-\frac 12 < b' \leq 0 \leq b \leq 1+b'$. Then we have
 \begin{align}
\bigg\| \int_{0}^{t} S(t-t')F(t')dt' \bigg\|_{X^{s,b}_{T}} \les T^{1+b'-b} \|F\|_{X^{s,b'}_{T}}, \label{duhamelest}
\end{align}
 for any $0<T\leq 1$. 
\end{enumerate}
  
\end{lemma}

We now state the crucial nonlinear estimate for the operators of \eqref{nonlinops}.

\begin{proposition}\label{prop:nonlinest}
Let $\al >1$ and $s\geq 0$. Then for $b>\frac{1}{2}$, there exists $b'<\frac 12$, sufficiently close to $\frac{1}{2}$, such that we have
\begin{align}
\| \mathcal{R}(v_1,v_2,v_3)\|_{X^{s,-b'}_{T}} +\|\mathcal{N}(v_1,v_2,v_3)\|_{X^{s,-b'}_{T}} \les \min_{k\in \{1,2,3\} }\bigg(  \|v_k\|_{X^{s,b}_{T}}  \prod_{\substack{j=1 \\ j\neq k}  }^{3}\|v_j\|_{X^{0,b}_{T}}    \bigg) \label{nonlinestB}
\end{align}
for any $0<T\leq 1$.
\end{proposition}

\begin{proof}
We follow the proof in \cite[Proposition 5]{demirbas2013existence}, but since $\al >1$, the proof here is  simpler. 
\noi
Let $\tilde{v}_j$ be extensions of $v_j$, for $j=1,2,3$. Then it suffices to prove 
\begin{align}
\| \mathcal{N}(\tilde{v}_1,\tilde{v}_2, \tilde{v}_3)\|_{X^{s,-b'}} +\|  \mathcal{R}(\tilde{v}_1,\tilde{v}_2, \tilde{v}_3)\|_{X^{s,-b'}} \les \min\bigg(  \|\tilde{v}_k\|_{X^{s,b}}  \prod_{\substack{j=1 \\ j\neq k}  }^{3}\|\tilde{v}_j\|_{X^{0,b}}    \bigg), \label{b1}
\end{align}
since \eqref{nonlinestB} follows from \eqref{b1} by taking an infimum over all extensions $\tilde{v}_j$ of $v_j$. For simplicity, we write $\tilde{v}_j$ as $v_j$. 
We begin with the estimate for $\mathcal{N}$. 
By Plancherel, we have 
\begin{align*}
\|\mathcal{N}(v_1,v_2,v_3)\|_{X^{s,-b'}}= \bigg\| \intt_{\tau=\tau_1-\tau_2+\tau_3} \sum_{\G(n)} \frac{\jb{n}^{s} \ft{v}_1(\tau_1,n_1)\cj{\ft v_2 (\tau_2,n_2)} \ft{v}_3(\tau_3,n_3)   }{\jb{\tau-|n|^{2\al}}^{b'} }d\tau_1 d\tau_2 \bigg\|_{L^{2}_{\tau}\l^2_n}.
\end{align*}
Notice $|n|\les \max_{j=1,2,3}|n_j|$ and thus we may suppose, without loss of generality, $|n|\les |n_1|$. 
 We define
 \begin{align*}
 f_1(\tau,n)=\jb{n}^{s}\jb{\tau -|n|^{2\al}}^{b}|\ft v_1(\tau,n)|\,\,\, \text{and} \,\,\, g_j(\tau,n)=\jb{\tau -|n|^{2\al}}^{b}|\ft v_j(\tau,n)|\,\,\, \text{for} \,\,\, j=2,3.
\end{align*}
  Then we see that \eqref{b1} follows if we prove 
 \begin{align}
 \begin{split}
\bigg\| \intt_{\tau=\tau_1-\tau_2+\tau_3} \sum_{\G(n)} \frac{ f_1(\tau_1,n_1) g_2 (\tau_2,n_2)g_3(\tau_3,n_3)   }{\jb{\tau-|n|^{2\al}}^{b'}}& \prod_{j=1}^{3}\frac{1}{\jb{\tau_j-|n_j|^{2\al}}^{b}} d\tau_1 d\tau_2 \bigg\|_{L^{2}_{\tau}\l^2_n} \\
& \les \|g_2\|_{L^2_{\tau}\l^2_n}  \|g_3\|_{L^2_{\tau}\l^2_n}\|f_1\|_{L^2_{\tau}\l^2_n}.
\end{split}
\label{b2}
\end{align} 
By duality, we have
\begin{align*}
\text{LHS of} \,\, \eqref{b2} \les \sup_{\|h\|_{L^{2}_{\tau}\l^2_n}=1}\intt_{\tau-\tau_1+\tau_2-\tau_3=0} \sum_{\substack{n-n_1+n_2-n_3=0 \\ n_1,n_3\neq n}} &\frac{ f_1(\tau_1,n_1) g_2 (\tau_2,n_2)g_3(\tau_3,n_3) h(\tau,n)}{\jb{\tau-|n|^{2\al}}^{b'}}\\
& \times  \prod_{j=1}^{3}\frac{1}{\jb{\tau_j-|n_j|^{2\al}}^{b}} d\tau_1 d\tau_2 d\tau.
\end{align*}
Applying Cauchy-Schwarz, we bound this by
\begin{align*}
\|h\|_{L^{2}_{\tau}\l^{2}_{n}} \|g_2\|_{L^2_{\tau}\l^2_n}  \|g_3\|_{L^2_{\tau}\l^2_n}\|f_1\|_{L^2_{\tau}\l^2_n}  \sup_{n,\tau} M_{n,\tau}^{\frac 12},
\end{align*}
where 
\begin{align*}
 M_{n,\tau}= \sum_{\substack{n_1-n_2+n_3=n \\ n_1,n_3\neq n}}  \intt_{\tau=\tau_1-\tau_2 +\tau_3} \frac{1}{\jb{\tau-|n|^{2\al}}^{2b'}}  \prod_{j=1}^{3}\frac{1}{\jb{\tau_j-|n_j|^{2\al}}^{2b}} d\tau_1 d\tau_2.  \end{align*}
It is then clear that we have proved the thesis if we show 
\begin{align*}
\sup_{n,\tau} M_{n,\tau} <\infty.
\end{align*}
Fix $n\in \Z$ and $\tau \in \R$. 
As $2b>1$, integrating in $\tau_2$ and $\tau_1$ (by using Lemma~\ref{lemma:sumestimate} twice) gives 
\begin{align*}
M_{n,\tau} \les  \sum_{\substack{n_1-n_2+n_3=n \\ n_1,n_3\neq n}} \frac{1}{\jb{\tau -|n|^{2\al}}^{2b'}} \frac{1}{ \jb{ \tau -|n|^{2\al} -\phi(\cj n)}^{2b}} \les  \sum_{\substack{n_1-n_2+n_3=n \\ n_1,n_3\neq n}} \frac{1}{\jb{\phi(\cj n)}^{2b'}}.
\end{align*}
The last inequality above follows from the triangle inequality.
Recalling that $n_1,n_3\neq n$, Lemma~\ref{lemma: phase lower bound} implies $|\phi(\cj n)| \gtrsim 1$ and hence by Lemma~\ref{lemma:sumestimate}, we have
\begin{align*}
M_{n,\tau} &\les  \sum_{\substack{n_1-n_2+n_3=n \\ n_1,n_3\neq n}}  \frac{1}{\jb{n-n_1}^{2b'}\jb{n-n_3}^{2b'} n_{\max}^{4b'(\al-1)}} \\
& \les \bigg( \sum_{n_1} \frac{1}{\jb{n-n_1}^{2b'}\jb{n_1}^{2b'(\al-1)} }\bigg)^{2} <C<\infty,
\end{align*}
independently of $n$ provided we choose $b'<\frac 12$ such that 
\begin{align*}
2b' \al>1.
\end{align*}
The above condition is satisfied since $\al>1$.
This completes the proof of \eqref{b1} for the non-resonant operator $\mathcal{N}$. The case for the resonant operator $\mathcal{R}$ is simpler. Indeed, by
 Young's inequality and H\"{o}lder's inequality, we have
\begin{align*}
\|\mathcal{R}(v_1,v_2,v_3)\|_{X^{s,-b'}} &\leq \|\mathcal{R}(v_1,v_2,v_3)\|_{X^{s,0}} =\big\| \jb{n}^{s} (\ft v_1 \ast_{\tau} \cj{\ft v_2} \ast_{\tau} \ft v_3)(\tau,n) \big\|_{L^2_{\tau}\l^2_n}   \\
& \les \big\| \jb{n}^{s} \| \ft v_1 (\tau,n) \|_{L^{2}_{\tau}} \prod_{j=2}^{3} \|\ft v_j(\tau,n) \|_{L^{1}_{\tau}} \big\|_{\l^{2}_n} \\
& \les \big\| \jb{n}^{s} \|\ft v_1(\tau,n)\|_{L^2_{\tau}} \big\|_{\l^2_n}   \prod_{j=2}^{3} \big\| \| \jb{\tau-|n|^{2\al}}^{b}\ft v_j(\tau,n) \|_{L^2_{\tau}} \big\|_{\l^2_n}^{2} \\
& \sim \|v_1\|_{X^{s,0}} \prod_{j=2}^{3} \|v_j\|_{X^{0,b}} \les \|v_1\|_{X^{s,b}} \prod_{j=2}^{3} \|v_j\|_{X^{0,b}}.
\end{align*}
It is clear that the right hand side of \eqref{nonlinestB} follows analogously. 
 This completes the proof.
\end{proof}

We now prove Proposition~\ref{prop:appBlocal}.

\begin{proof}[Proof of Proposition~\ref{prop:appBlocal}]  
Fix $\al>1$ and let $u_0\in L^2 (\T)$. For $0<T\leq 1$, from \eqref{Lambda}, we have
\begin{align*}
\Lambda(v)(t)=\Lambda(v)_{u_0}(t):= S(t)u_0-i \int_{0}^{t} S(t-t') ( \mathcal{N}(v)(t')+\mathcal{R}(v)(t')) dt'.
\end{align*}
Now let $b'$ be given by Proposition~\ref{prop:nonlinest} and write $b'=\tfrac{1}{2-\dl}$ for some small $\dl>0$ and then set $b=\tfrac 12 +\tfrac{\dl}{2}$.
Then by \eqref{linest}, \eqref{duhamelest} and \eqref{nonlinestB}, 
we have 
\begin{align}
\|\Lambda(v)\|_{X^{0,b}_{T}} & \les \|u_0\|_{L^2} + T^{\frac{\dl}{2} }\|v\|_{X^{0,b}_{T}}^{3}. \label{contract1}
\end{align} 
Similarly,
\begin{align}
\|\Lambda(v_1)-\Lambda(v_2)\|_{X^{0,b}_{T}} \les T^{\frac{\dl}{2}}( \|v_1\|_{X^{0,b}_{T}}^{2}+\|v_2\|_{X^{0,b}_{T}}^{2}) \|v_1-v_2\|_{X^{0,b}_{T}}. \label{contract2}
\end{align}
With $R\sim 2\|u_0\|_{L^{2}}$, we let $B_{R}$ be the closed ball of radius $R$ in $X^{0,b}_{T}$. It follows by the contraction mapping theorem using \eqref{contract1} and \eqref{contract2} and choosing $T=T(\|u_0\|_{L^2})>0$ that $\Lambda$ is a contraction over $B_{R}$ and hence we have a unique fixed point $v\in X^{0,b}_{T}$.

Now for $s>0$, let $u_0 \in H^{s}(\T)$. Then by \eqref{linest}, \eqref{duhamelest} and \eqref{nonlinestB}, 
we obtain 
\begin{align}
\begin{split}
\|\Lambda(v)\|_{X^{s,b}_{T}} & \les \|u_0\|_{H^s} + T^{\frac{\dl}{2} }\|v\|_{X^{0,b}_{T}}^{2}\|v\|_{X^{s,b}_{T}}  \\
& \les \|u_0\|_{H^s} + T^{\frac{\dl}{2}  }\|u_0\|_{L^2}^{2}\|v\|_{X^{s,b}_{T}}, 
\end{split}
\label{contract3}
\end{align}
and
\begin{align}
\| \Lambda(v_1)-\Lambda(v_2) \|_{X^{s,b}_{T}} & \les T^{\frac{\dl}{2}} \|u_0\|_{L^{2}}^{2}\|v_1-v_2\|_{X^{s,b}_{T}}.\label{contract4}
\end{align}
 It is clear from \eqref{contract3} and \eqref{contract4} that we then obtain a unique solution $v\in X^{s,b}_{T}$ for the same $T=T(\|u_0\|_{L^{2}})>0$. Continuity of $v$ in time follows from \eqref{ctsembed}.
This completes the local well-posedness proof.
\end{proof}

We now move onto the proof of the global well-posedness of \eqref{FNLS} for any $\al>1$ as stated in Proposition~\ref{prop: gwp}~(i).

\begin{proof}[Proof of Proposition~\ref{prop: gwp}~(i) for $\al>1$]
When $s\geq \al$, we have for smooth solutions to \eqref{FNLS1} the growth bound 
\begin{align*}
\|u(t)\|_{H^{s}} \leq e^{Ct} \|u_0\|_{H^{s}},
\end{align*} with $C=C(\|u_0\|_{H^1})$.
This follows from the coercivity of the energy\footnote{In the focusing case, one uses the Gagliardo-Nirenberg inequality \eqref{GNineq}, as discussed in the introduction.}, the inequality 
\begin{align*}
\| |u|^2 u\|_{H^{s}} \les \|u\|_{H^{\al}}^{2}\|u\|_{H^{s}}, \qquad (s\geq \al)
\end{align*}
and Gronwall's inequality. 

For $s=0$, we have the mass conservation $\|v(t)\|_{L^{2}}=\|u_0\|_{L^{2}}$ for any $t\in \R$, immediately yielding global-in-time existence.  For $0<s <\al$, we iterate the mass conservation. 
The $L^{2}$ local theory (Proposition~\ref{prop:appBlocal}) implies there is a $T_0$ depending on $\|u_0\|_{L^{2}}$ such that 
\begin{align}
\| v\|_{X^{0,b}_{T_0}} \les \|u_0\|_{L^{2}}, \label{b5}
\end{align}
for, say, $b=\frac 12 +$. 
We have from \eqref{contract3} and \eqref{b5}, that for any $T_1\leq T_0$, 
\begin{align*}
\|v\|_{X^{s,b}_{T_1}} \les \|u_0\|_{H^s} +T_1^{\ta}\|u_0\|_{L^{2}}^{2}\|v\|_{X^{s,b}_{T_1}}.
\end{align*}
Therefore, there is a $T_2(\|u_0\|_{L^2}) \leq T$ such that 
\begin{align*}
\|v\|_{X^{s,b}_{T_2}} \les \|u_0\|_{H^{s}}.
\end{align*}
From \eqref{ctsembed}, we get the a priori bound 
\begin{align*}
\sup_{t\in [0,T_2]} \|u(t)\|_{H^{s}} \les \|u_0\|_{H^{s}},
\end{align*}
which can be iterated to yield 
\begin{align*}
\sup_{t\in [0,T]}\|u(t)\|_{H^{s}} \les e^{K^{\frac{2}{\ta}}T}\|u_0\|_{H^s},
\end{align*}
for any $T>0$ and any $u_0\in H^{s}(\T)$ with $\|u_0\|_{L^{2}}\leq K$.
\end{proof}

\end{appendix}

\begin{acknowledgment}
\rm 
J.\,F. and W.\,J.\,T. would like to thank their advisors, Tadahiro Oh and Oana Pocovnicu, for suggesting this problem and their continuous guidance and support throughout.
J.\,F. and W.\,J.\,T. were supported by The Maxwell Institute Graduate School in Analysis and its
Applications, a Centre for Doctoral Training funded by the UK Engineering and Physical
Sciences Research Council (grant EP/L016508/01), the Scottish Funding Council, Heriot-Watt
University and the University of Edinburgh.
J.\,F. also acknowledges support from Tadahiro Oh's ERC starting grant (no. 637995 “ProbDynDispEq”). 
\end{acknowledgment}

\end{document}